\documentclass[12pt, a4paper]{amsart}

\usepackage[hmargin=30mm, vmargin=25mm, includefoot, twoside]{geometry}
\usepackage[bookmarksopen=true]{hyperref}

\usepackage{amsfonts,amssymb,verbatim}

\usepackage{latexsym}
\usepackage{mathrsfs}
\usepackage{stmaryrd}
\usepackage{xspace}
\usepackage{enumerate, paralist}
\usepackage{graphicx}
\usepackage[all]{xy}
\usepackage{extarrows}
\usepackage{tabu}
\usepackage{accents}
\usepackage{tensor}

\usepackage[usenames,dvipsnames]{color}

\usepackage{txfonts, pxfonts}

\usepackage{amsmath}

\usepackage[tbtags]{mathtools}

\usepackage{amsthm, thmtools}

\newtheorem{thm}{Theorem}[section]
 \newtheorem{cor}[thm]{Corollary}
 \newtheorem{lem}[thm]{Lemma}
 \newtheorem{prop}[thm]{Proposition}

\newtheorem{introthm}{Theorem}

\numberwithin{equation}{section}

 \theoremstyle{definition}
  \newtheorem{defn}[thm]{Definition}
  \newtheorem{question}{Question}[section]

 \theoremstyle{remark}
 \newtheorem{rem}[thm]{Remark}

\newtheorem*{claim*}{Claim}

\def\NN{\mathbb{N}}
\def\CC{\mathbb{C}}
\def\RR{\mathbb{R}}
\def\ZZ{\mathbb{Z}}

\def\B{\mathfrak{B}}

\def\M{\mathfrak{M}}
\def\KK{\mathfrak{K}}

\def\H{\mathcal{H}}
\def\K{\mathcal{K}}
\def\N{\mathcal{N}}

\def\G{\mathcal{G}}
\def\L{\mathcal{L}}
\def\T{\mathcal{T}}
\def\Gz{\mathcal{G}^{(0)}}

\def\A{\mathscr{A}}

\def\SOT{\mathrm{(SOT)}}
\def\Id{\mathrm{Id}}
\def\I{\mathrm{I}}
\def\Tr{\mathrm{Tr}}

\def\ppg{\mathrm{prop}}
\def\ran{\mathrm{Ran}}
\def\ker{\mathrm{Ker}}
\def\dim{\mathrm{dim}}

\def\rank{\mathrm{rank}}
\def\supp{\mathrm{supp}}

\def\girth{\mathrm{girth}}
\def\med{\mathrm{med}}

\def\m{\mathbf{m}}
\def\r{\mathbf{r}}

\begin{document}

\title{$K$-Theoretic Comparison of Roe and Quasi-Local Algebras via Projections}
\author{Kang Li}
\author{Jiawen Zhang}
\author{Jingming Zhu}

\address[Kang Li]{Centre for Mathematical Sciences, Lund University, Box 118, SE-221 00 Lund, Sweden.}
\email{kang.li@math.lth.se}

\address[Jiawen Zhang]{School of Mathematical Sciences, Fudan University, 220 Handan Road, Shanghai, 200433, China.}
\email{jiawenzhang@fudan.edu.cn}

\address[Jingming Zhu]{College of Data Science, Jiaxing University, 899 Guangqiong Road, Jiaxing, 314000, China.}
\email{jingmingzhu@zjxu.edu.cn}

\thanks{Jiawen Zhang was partly supported by the National Key R{\&}D Program of China 2022YFA100700 and NSFC 12422107.}

\thanks{Keywords: $K$-theory, Roe and quasi-local algebras, expanders, ghost operators, rank profiles.}

\begin{abstract}
A central question in higher index theory and operator algebras is whether the Roe algebra and the quasi-local algebra associated with a metric space of bounded geometry coincide, or at least have the same $K$-theory.

In this paper, we focus on a \emph{sparse} metric space $X$. We prove the following three main results: (1) For a block-diagonal operator $T$ with uniformly bounded block-rank, $T$ is quasi-local if and only if it is in the Roe algebra. (2) In general, we discover a ghost block-diagonal projection which is quasi-local but not in the Roe algebra. (3) For a sequence of expander graphs with sufficiently large girth, the inclusion of the uniform Roe algebra into the uniform quasi-local algebra induces a \emph{non-surjective} map on their $K_0$-groups. This yields the first known $K$-theoretic distinction between the uniform Roe algebra and the uniform quasi-local algebra.
\end{abstract}

\date{\today}

\maketitle

\section{Introduction}\label{sec:intro}

Roe algebras are $C^*$-algebras associated to metric spaces, which encode their coarse geometric information. They were introduced by John Roe in \cite{Roe88}, where he showed that their $K$-theory can serve as receptacles for higher indices of elliptic differential operators on open manifolds. Hence, the computation of the $K$-theory of Roe algebras becomes crucial in the study of higher index theory, and has led to significant progress in topology, geometry, and analysis (\emph{e.g.}, \cite{Roe93, Roe96}). There is also a uniform version of the Roe algebra, which equally plays a key role in higher index theory (see \cite{Spa09}).

Although the Roe algebra is important, it is usually hard to tell whether a given operator belongs to it or not. To overcome this issue, John Roe introduced the notion of quasi-local operator in \cite{Roe88}. The set of all quasi-local operators forms a $C^*$-algebra, called the quasi-local algebra, which is formally larger than the Roe algebra. As shown by Engel in \cite{Eng14}, the higher indices of elliptic pseudo-differential operators on open manifolds naturally live in the $K$-theory of the quasi-local algebra. Hence, the following question is crucial and has drawn considerable attention over the last few years:

\begin{question}\label{MainQues}
When does the Roe algebra coincide with the quasi-local algebra? 
Does the inclusion between them induce an isomorphism on their $K$-theories?
\end{question}

To be more precise, let us explain the notions for the uniform case. Let $(X,d)$ be a discrete metric space of bounded geometry, and $T \in \B(\ell^2(X))$. Define the \emph{propagation} of $T$ to be $\ppg(T):=\sup\{d(x,y): T_{x,y} \neq 0\}$, where $T_{x,y}:=\langle T\delta_y, \delta_x \rangle$. The \emph{uniform Roe algebra of $X$} is defined to be the norm closure of all operators with finite propagation in $\B(\ell^2(X))$, denoted by $C^*_u(X)$. On the other hand, $T$ is called \emph{quasi-local} if for any $\varepsilon>0$, there exists $R>0$ such that for any $A, B \subseteq X$ with $d(A,B) > R$, $\|\chi_A T \chi_B\| < \varepsilon$. The \emph{uniform quasi-local algebra of $X$} is defined to be the set of all quasi-local operators in $\B(\ell^2(X))$, denoted by $C^*_{uq}(X)$. By definition, $C^*_u(X) \subseteq C^*_{uq}(X)$ holds trivially and Question~\ref{MainQues} above concerns the opposite direction. More generally, for a Hilbert space $\H_0$, see Section~\ref{sec:preliminaries} for the definitions of the Roe algebra $C^*(X;\H_0)$ and the quasi-local algebra $C^*_q(X;\H_0)$.

On the positive side, the best known result is due to \v{S}pakula and the second-named author. In \cite{SZ20}, they showed that if the metric space $X$ has Yu's Property A (introduced in \cite{Yu00} and simplified in \cite{NWZ25, ZZ25}), then $C^*_u(X) = C^*_{uq}(X)$. (See \cite{Oza25} for an alternative proof.) This work has led to a number of applications, including the classification of Cartan subalgebras in Roe algebras \cite{WW20}, results on the rigidity problem \cite{BBFKVW22} and higher index theory for elliptic pseudo-differential operators on open manifolds \cite{Eng19}.
On the negative side, Ozawa recently showed in \cite{Oza25} that if $X$ contains a sequence of asymptotic expanders (introduced in \cite{LNSZ21}, which generalises the classical notion of expanders), then $C^*_u(X) \neq C^*_{uq}(X)$.

Concerning their $K$-theories, there are no direct answers to the best of our knowledge, except for two closely related works \cite{BCZ23, CGZ24} due to the second-named author and his collaborators. They introduced an intermediate operator algebra (called the strongly quasi-local algebra) between the Roe algebra and the quasi-local algebra, and showed that if the space can be coarsely embedded into some Hilbert space, then the Roe algebra and the strongly quasi-local algebra have the same $K$-theories. However, it is unclear whether the quasi-local algebra and the strongly quasi-local algebra always share the same $K$-theories.

To simplify the $K$-theory part of Question~\ref{MainQues}, it is sufficient to focus on sparse metric spaces. More precisely, a metric space $(X,d)$ is called \emph{sparse} if there exists a sequence of finite subspaces $\{(X_n,d_n)\}_{n\in \NN}$ such that $X = \bigsqcup_{n\in \NN} X_n$ and
\[
d(X_n, X_m) \to \infty \quad \text{as} \quad n+m \to \infty \quad \text{and} \quad n \neq m.
\]
In this case, $X$ is also called a \emph{coarse disjoint union} of $\{(X_n,d_n)\}_{n}$.
Using annulus decomposition, any metric space can be coarsely decomposed into a union of two sparse subspaces. Applying the coarse Mayer-Vietoris sequence \cite{BCZ23, HRY93}, the $K$-theory question can be reduced to the case of sparse spaces.

Following this philosophy, the first two authors and their collaborators have studied block-rank-one projections on sparse spaces. More precisely, for $X = \bigsqcup_{n\in \NN} X_n$, consider a rank-one projection $P_n$ on $\ell^2(X_n)$ and set $P:=\mathrm{(SOT)}$-$\sum_{n} P_n$. As shown in \cite{KLVZ21, LSZ23}, $P \in C^*_u(X)$ if and only if $P \in C^*_{uq}(X)$, which leads to new counterexamples to the coarse Baum-Connes conjecture. The proof follows a graph-theoretic approach by introducing the notion of measured asymptotic expanders (see also \cite{LSZ24}), showing $P \in C^*_{uq}(X)$ is equivalent to $\{X_n\}_n$ being measured asymptotic expanders. Meanwhile, parallel results have also been established in the setting of dynamics \cite{LVZ23,LVZ23b} and groupoids \cite{JZZ23, LWZ25}.

In this paper, we concentrate on higher-rank block-diagonal operators in the Roe algebra and the quasi-local algebra of sparse metric spaces as well as their $K$-theories. We provide the first example where the inclusion of the uniform Roe algebra into the uniform quasi-local algebra induces a \emph{non-surjective} map on their $K_0$-groups. Our route begins with a systematic study of block-diagonal operators, which reveals a dichotomy between bounded and unbounded block-rank. Let us introduce our results in two parts.

\subsection{Rank dichotomy for block-diagonal quasi-local operators}

First, we study block-diagonal operators with uniformly bounded block-rank. The main result is as follows.

\begin{introthm}\label{thm:main result projection}
 Let $\{(X_n,d_n)\}_{n\in \NN}$ be a sequence of finite metric spaces of uniformly bounded geometry, and $X:=\bigsqcup_{n\in \NN} X_n$ be their coarse disjoint union. Let $M \in \NN$ and $\H_0$ be a Hilbert space. For $T=\SOT$-$\sum_{n\in \NN} T_n$ in $\B(\ell^2(X;\H_0))$ with $T_n\in \B(\ell^2(X_n;\H_0))$ having rank at most $M$, we have $T \in C^*(X;\H_0)$ \emph{if and only if} $T \in C^*_q(X;\H_0)$.
\end{introthm}

Note that when $M=1$ and each $T_n$ is a projection, Theorem~\ref{thm:main result projection} recovers \cite[Theorem C]{KLVZ21} and \cite[Theorem 6.1 and Theorem 6.6]{LSZ23}.
The key ingredient in the proof is the following decomposition result (see Theorem~\ref{thm:decomposition of projections} for the projection case and Theorem~\ref{thm:peeling} for the general case):

\begin{introthm}\label{introthm:decomposition}
Let $\{(X_n,d_n)\}_{n\in \NN}$ be a sequence of finite metric spaces of uniformly bounded geometry, and $X:=\bigsqcup_{n\in \NN} X_n$ be their coarse disjoint union. Let $M \in \NN$ and $\H_0$ be a Hilbert space. Let $P=\SOT$-$\sum_{n\in \NN} P_n$ be a projection in $C^*_q(X;\H_0)$, where each $P_n \in \B(\ell^2(X_n;\H_0))$ and $1 \leq \rank(P_n) \leq M$. Then there is a block-rank-one projection $Q=\SOT$-$\sum_{n\in \NN} Q_n$ in $C^*_q(X;\H_0)$ such that $Q \leq P$. Hence, $P$ can be decomposed as $P=P^{(1)}+\cdots+P^{(M)}$, where each $P^{(i)}\in C^*_q(X;\H_0)$ is a projection with block-rank at most one. A similar result holds for general block-rank-$M$ operators.
\end{introthm}

According to \cite[Theorem 6.6]{LSZ23}, block-rank-one quasi-local projections belong to the Roe algebra. Hence the projection case in Theorem~\ref{thm:main result projection} follows directly from Theorem~\ref{introthm:decomposition}. On the other hand, Theorem~\ref{introthm:decomposition} shows that if a block-rank-$M$ projection $P$ belongs to the Roe algebra, then it contains a block-rank-one sub-projection, which also sits in the Roe algebra. This provides an affirmative answer to \cite[Question 5.31]{Vig26} and a partial answer to \cite[Question 6.1]{BFV24}.

Let us briefly explain the idea of the proof for Theorem~\ref{introthm:decomposition}. First, as in the block-rank-one case, for a discrete metric space $X$, a rank-$M$ projection $P$ in $\B(\ell^2(X))$ and $A,B \subseteq X$, we compute
\[
    \|\chi_A P \chi_B\| = \left\|\m(A)^{1/2} \m(B)^{1/2}\right\|.
\]
Here $\m(A)$ is the Gram matrix of the restrictions of an orthonormal basis of $\ran(P)$. This turns quasi-locality into a finite-dimensional spectral condition: a bad block forces a uniform spectral gap, while the uniform rank bound allows us to choose finitely many such gaps and partition all bad blocks accordingly. Continuous functional calculus then produces a quasi-local sub-projection of strictly smaller rank on every bad block, and iterating this reduction yields a genuine block-rank-one quasi-local sub-projection.

The uniform bound on the block-ranks is essential. We now turn to block-diagonal projections with unbounded ranks. We show that they behave differently from those with bounded ranks and can be used to distinguish the uniform Roe and quasi-local algebras. More precisely, we show the following:

\begin{introthm}\label{prop:infiniteprojection intro}
 Let $\{X_n\}_{n \in \mathbb{N}}$ be a sequence of asymptotic expanders and $X = \bigsqcup_{n=1}^{\infty} X_n$ be their coarse disjoint union. Then there exists a ghost projection $P = \SOT$-$\sum_{n\in \NN} P_n$ where each $P_n$ is a projection on $\ell^2(X_n)$ such that $P \in C_{uq}^*(X)$, while $P \notin C_{u}^*(X)$.
\end{introthm}

The proof relies on Ozawa's recent work \cite{Oza25}, where he showed that the direct product of the matrix algebras $\prod_n \M_n(\CC)$ can never be embedded into any uniform Roe algebra as a $C^*$-subalgebra, while the product can be embedded into the uniform quasi-local algebra for the space $X$ above. Here $\M_n(\CC)$ denotes the $n$-by-$n$ complex matrix algebra.
We notice that the image of the embedded direct product of the matrix algebras in the uniform quasi-local algebra consists of ghost operators (see Definition~\ref{defn:ghost}). Using functional calculus for each matrix in a uniform way, we deduce Theorem~\ref{prop:infiniteprojection intro}. Thus the uniform bound on the block-ranks marks a genuine rigidity threshold for block-diagonal quasi-local operators.

\subsection{$K$-theoretic comparison of Roe and quasi-local algebras}

The dichotomy above shows that projections can be used to tell the difference between the uniform Roe algebra and the uniform quasi-local algebra. However, this is not sufficient to distinguish their $K$-theories: a projection in $C^*_{uq}(X)\setminus C^*_u(X)$ may still represent a $K_0$-class lying in the image of $K_0(C^*_u(X))$.

Our main $K$-theoretic result explores the structure of their $K$-theories and provides the first example where the inclusion $C^*_u(X) \hookrightarrow C^*_{uq}(X)$ does \emph{not} induce an isomorphism on their $K_0$-groups. This gives a negative answer to the $K$-theory part of Question~\ref{MainQues}. Recall that the girth of a graph $X$, denoted by $\girth(X)$, is the length of the shortest non-trivial cycle in $X$.

\begin{introthm}\label{thm:main result K-theory}
Let $\{X_n\}$ be a sequence of expander graphs such that there exists $\gamma>0$ satisfying $\girth(X_n)\geq\gamma\log |X_n|$ for sufficiently large $n\in \NN$, and $X$ be their coarse disjoint union. Then the map
\[
\iota_\ast: K_0(C^*_u(X)) \longrightarrow K_0(C^*_{uq}(X)),
\]
induced by the inclusion $\iota: C^*_u(X) \hookrightarrow C^*_{uq}(X)$, is \emph{not} surjective.
\end{introthm}

Thanks to \cite{LPS88}, such expander graphs do exist. More precisely, \cite{LPS88} implies that there exists a sequence of $k$-regular Ramanujan graphs $\{X_n\}_{n\in \NN}$ for some $k \in \NN$ (for example, $k=6$) with girth satisfying:
\[
\girth(X_n)\geq \log_{k-1}|X_n| \quad \text{for sufficiently large} \quad n \in \NN.
\]
This establishes non-vacuity of our assumptions. However, it is unclear to us whether Theorem~\ref{thm:main result K-theory} holds in the non-uniform case.

The proof of Theorem \ref{thm:main result K-theory} is built on two key ingredients. The first is a quotient isomorphism showing that, in the large girth setting, the difference between the Roe and quasi-local algebras lies in their ghost ideals. For a discrete metric space $X$, denote the set of ghost operators in the Roe algebra $C^*(X;\H_0)$ by $G(X;\H_0)$, and those in the quasi-local algebra $C^*_{q}(X;\H_0)$ by $G_{q}(X;\H_0)$. We show the following:

\begin{introthm}\label{thm:quotient isom intro}
Let $\{X_n\}$ be a sequence of graphs with bounded valency and large girth, $X$ be their coarse disjoint union, and $\H_0$ be a Hilbert space. Then the inclusion from $C^*(X;\H_0)$ to $C^*_{q}(X;\H_0)$ induces a $*$-isomorphism $C^*(X;\H_0)/G(X;\H_0) \cong C^*_{q}(X;\H_0)/G_{q}(X;\H_0)$.
\end{introthm}

To prove Theorem~\ref{thm:quotient isom intro}, we need a new construction of limit operators for quasi-local operators. Recall that in \cite{SW17}, the limit operator theory for \emph{rich} operators in the Roe algebra was developed, and later  extended to the uniform quasi-local algebra in \cite{GQW24}. Based on these works, we introduce a new limit operator construction for quasi-local operators in $C^*_q(X;\H_0)$ to compute the norm of elements in the quotient algebra $C^*_{q}(X;\H_0)/G_{q}(X;\H_0)$. Then using the Schur-multiplier estimates from \cite{HSS10}, we establish a uniform truncation argument to approximate quasi-local operators using finite propagation ones in the quotient norm. 

When $\H_0= \CC$, we provide an alternative proof for Theorem~\ref{thm:quotient isom intro} using the notion of groupoid quasi-local algebras introduced in \cite{JZZ23}. This provides the first example of a non-amenable groupoid $\G$ whose reduced groupoid $C^*$-algebra coincides with the equivariant quasi-local algebra. See Appendix~\ref{app:groupoid} for details.

The second ingredient is a quantitative distinction inside the ghost ideals, detected by rank profiles. The restriction of operators in $C^*_{uq}(X)$ to their diagonal part provides a map from $C^*_{uq}(X)$ to the direct product of matrix algebras $\M_{|X_n|}(\CC)$, which induces the following map on $K_0$-groups
\[
\rank: K_0(C^*_{uq}(X)) \longrightarrow K_0\left( \prod_n \M_{|X_n|}(\CC) \big/ \bigoplus_n \M_{|X_n|}(\CC) \right) \hookrightarrow \prod_n \ZZ \big/ \bigoplus_n \ZZ.
\]
We show the following (see Proposition~\ref{prop:Roe ghost} and Corollary~\ref{cor:ghost in ql} for details). Here we denote $G_u(X)=G(X;\CC)$ the ideal of ghost operators in $C^*_u(X)$.

\begin{introthm}\label{thm:rank profile}
Let $\{X_n\}$ be a sequence of expander graphs such that there is $\gamma>0$ satisfying $\girth(X_n)\geq\gamma\log |X_n|$ for sufficiently large $n\in \NN$, and $X$ be their coarse disjoint union.
\begin{enumerate}
 \item For every $z\in K_0(G_u(X))$, there are $C<\infty$ and $\alpha>0$ such that $\rank((\iota\circ i_u)_\ast(z))$
has a representative $(s_n)_n \in \prod_n \ZZ$ satisfying $|s_n|\leq C |X_n|^{\,1-\alpha}$ for each $n$. Here $i_u: G_u(X) \to C^*_u(X)$ and $\iota: C^*_u(X) \to C^*_{uq}(X)$ are the inclusion maps.
 \item There is a ghost projection $P \in C^*_{uq}(X)$ with $\rank(P_n) = \left\lfloor\frac{|X_n|}{\log |X_n|}\right\rfloor$ for large $n$.
\end{enumerate}
\end{introthm}

The proof of (1) relies on the lifting technique from \cite{WY12} (originally by Higson). As for (2), we prove a strengthened version of Ozawa's embedding of the product of matrix algebras into the uniform quasi-local algebra with an extra control on the dimension (see Proposition~\ref{prop:ql subspace} and Corollary~\ref{cor:ghost in ql}).

Combining Theorem~\ref{thm:quotient isom intro} with Theorem~\ref{thm:rank profile} and using a diagram chase, we finish the proof of Theorem~\ref{thm:main result K-theory}.

Finally, Theorem~\ref{thm:quotient isom intro} also yields a sufficient condition on injectivity:

\begin{introthm}\label{thm:inj intro}
Let $\{X_n\}$ be a sequence of graphs with bounded valency and large girth, $X$ be their coarse disjoint union, and $\H_0$ a separable Hilbert space. If $X$ can be coarsely embedded into some Hilbert space, then the map $\iota_\ast: K_0(C^*(X;\H_0)) \longrightarrow K_0(C^*_{q}(X;\H_0))$, induced by the inclusion $\iota: C^*(X;\H_0) \hookrightarrow C_q^*(X;\H_0)$, is injective.
\end{introthm}

Such graphs do exist, \emph{e.g.}, the box spaces of the free group constructed in \cite{AGS12}. The condition of large girth is proved in \cite[Lemma 4.4]{AGS12}. However, the surjectivity of the map $\iota_\ast$ in Theorem~\ref{thm:inj intro} is still unclear to us. On the other hand, we currently do not know any example where the injectivity of $\iota_\ast$ fails.

\subsection*{Organisation} The paper is organised as follows: In Section~\ref{sec:preliminaries}, we collect background knowledge. Section~\ref{sec:decomposition} is devoted to the proof of Theorem~\ref{thm:main result projection}, where we prove Theorem~\ref{introthm:decomposition} in Section~\ref{ssec:proof of decomposition}. In Section~\ref{sec:unbounded case}, we prove Theorem~\ref{prop:infiniteprojection intro}. Section~\ref{sec:quotient isom} and Section~\ref{sec:rank profiles} are devoted to the proof of Theorem~\ref{thm:main result K-theory}, where we prove Theorem~\ref{thm:quotient isom intro} in Section~\ref{sec:quotient isom} and Theorem~\ref{thm:rank profile} in Section~\ref{sec:rank profiles}. Finally, we prove Theorem~\ref{thm:inj intro} in Section~\ref{sec:injectivity}.

\subsection*{Acknowledgements}
We would like to thank Jianchao Wu for helpful discussions on Lemma~\ref{lem:norm of APB} and Theorem~\ref{prop:infiniteprojection intro}.

\subsection*{AI use statement}
We gratefully acknowledge ChatGPT 5.6-Sol for suggesting ideas related to Theorem~\ref{introthm:decomposition} and Theorem~\ref{thm:main result K-theory}. Theorem~\ref{introthm:decomposition} grew out of an approximation decomposition from an earlier version of the paper. All mathematical arguments were carefully checked and substantially revised by the authors, who bear full responsibility for the final version.

\section{Preliminaries}\label{sec:preliminaries}

For a discrete metric space $(X,d)$, $x\in X$ and $R \geq 0$, denote
\[
B(x,R):=\{y\in X:d(x,y) \leq R\}.
\]
We say that $(X,d)$ has \emph{bounded geometry} if $\sup_{x\in X} |B(x,R)|$ is finite for each $R\geq 0$. For $A \subseteq X$ and $R \geq 0$, denote the \emph{$R$-neighbourhood of $A$} by
\[
\N_R(A):=\{x\in X: d(x,A) \leq R\},
\]
and the \emph{$R$-boundary of $A$} by $\partial_R(A):=\N_R(A) \setminus A$.

A subset $K \subseteq X \times X$ is called a \emph{partial translation} if there exists $S>0$ such that $K \subseteq \{(x,y) \in X \times X: d(x,y) \leq S\}$ and for each $x\in X$, there is at most one $y\in X$ with $(x,y) \in K$ and at most one $z\in X$ with $(z,x) \in K$. In this case, there is a bijection $t: D \to R$ with $D,R \subseteq X$ such that $K$ coincides with the graph of $t$.

Given a Hilbert space $\H_0$, denote the Hilbert space consisting of all square summable functions from $X$ to $\H_0$ by $\ell^2(X;\H_0) \cong \ell^2(X) \otimes \H_0$. Denote the set of all bounded linear operators on this Hilbert space by $\B(\ell^2(X;\H_0))$. For $T \in \B(\ell^2(X;\H_0))$ and $x,y \in X$, define $T_{x,y} \in \B(\H_0)$ by
\[
\langle T_{x,y}\xi,\eta \rangle = \langle T(\delta_y \otimes \xi), \delta_x \otimes \eta \rangle \quad \text{for} \quad \xi, \eta \in \H_0.
\]
Here $\delta_x$ denotes the Dirac function at $x$ in $\ell^2(X)$.

For a closed subspace $U\subseteq \ell^2(X;\H_0)$, denote by $P_U \in \B(\ell^2(X;\H_0))$ the orthogonal projection onto $U$. Denote $u_A:=\chi_A u$ for $u\in \ell^2(X;\H_0)$ and $A \subseteq X$, and $P_u:=P_{\CC u}$ if $\|u\|=1$. Here $\chi_A$ denotes the characteristic function of $A$.

A sequence of finite metric spaces $\{(X_n,d_n)\}_{n\in \NN}$ is said to have \emph{uniformly bounded geometry} if $\sup_n \sup_{x\in X_n} |B(x,R)|$ is finite for each $R\geq 0$. Recall that their \emph{coarse disjoint union} is a metric space $(X,d)$ where $X = \bigsqcup_{n\in \NN} X_n$ as a set, the restriction of $d$ to each $X_n$ is $d_n$ and $d$ satisfies $d(X_n, X_m) \to \infty$ as $n+m \to \infty$ and $n \neq m$.
Clearly, $\{(X_n,d_n)\}_{n\in \NN}$ has uniformly bounded geometry \emph{if and only if} their coarse disjoint union has bounded geometry.
For such $X$ and a Hilbert space $\H_0$, an operator $T\in \B(\ell^2(X;\H_0))$ is called \emph{block-diagonal} if $T=\SOT$-$\sum_{n\in \NN} T_n$, where each $T_n \in \B(\ell^2(X_n;\H_0))$. Such $T$ is said to have \emph{block-rank at most $M$} if each $T_n$ has rank at most $M$, and \emph{block-rank-$M$} if each $T_n$ has rank exactly $M$.

\begin{defn}\label{defn:ppg and quasi-local}
Let $(X,d)$ be a discrete metric space, $\H_0$ be a Hilbert space and $T \in \B(\ell^2(X;\H_0))$.
\begin{enumerate}
  \item The \emph{propagation} of $T$ is defined to be $\ppg(T):=\sup\{d(x,y): T_{x,y} \neq 0\}$.
  We say that $T$ has \emph{finite propagation} if $\ppg(T) < \infty$.
  \item Given $\varepsilon, R>0$, say that $T$ is \emph{$(\varepsilon,R)$-quasi-local} if for any $A,B \subseteq X$ with $d(A,B) > R$, we have $\|\chi_A T \chi_B\| < \varepsilon$. We say that $T$ is \emph{quasi-local} if for any $\varepsilon>0$, there exists $R>0$ such that $T$ is $(\varepsilon,R)$-quasi-local.
  \item We say that $T$ is \emph{locally compact} if each $T_{x,y}$ is a compact operator on $\H_0$.
\end{enumerate}
\end{defn}

\begin{defn}\label{defn:Roe and quasi-local algebra}
Let $(X,d)$ be a discrete metric space of bounded geometry and $\H_0$ a Hilbert space.
\begin{enumerate}
  \item The \emph{algebraic Roe algebra} of $X$ (with coefficient $\H_0$), denoted by $\CC[X;\H_0]$, is the set of all finite propagation and locally compact operators in $\B(\ell^2(X;\H_0))$. The \emph{Roe algebra} of $X$ (with coefficient $\H_0$), denoted by $C^*(X;\H_0)$, is the norm closure of $\CC[X;\H_0]$.
  \item The \emph{quasi-local algebra} of $X$ (with coefficient $\H_0$), denoted by $C^*_q(X;\H_0)$, is the set of all quasi-local and locally compact operators in $\B(\ell^2(X;\H_0))$.
\end{enumerate}
\end{defn}

When $\H_0 = \CC$, simply denote $\CC_u[X]:=\CC[X;\CC]$ and $C^*_u(X) = C^*(X;\CC)$, where the latter is the uniform Roe algebra of $X$. Also $C^*_{uq}(X)=C^*_q(X;\CC)$ is the uniform quasi-local algebra.
When $\H_0$ is infinite-dimensional and separable, $C^*(X;\H_0)$ is the usual \emph{Roe algebra of $X$}, simply denoted by $C^*(X)$, and $C^*_q(X;\H_0)$ is the usual \emph{quasi-local algebra of $X$}, simply denoted by $C^*_{q}(X)$.

Finally, we recall the notion of ghost operators:

\begin{defn}[G. Yu]\label{defn:ghost}
Let $(X,d)$ be a discrete metric space of bounded geometry and $\H_0$ a Hilbert space. Say that $T \in \B(\ell^2(X;\H_0))$ is a \emph{ghost} operator if for any $\varepsilon>0$, there exists a finite subset $F \subseteq X$ such that $\|T_{x,y}\| < \varepsilon$ for any $(x,y) \in (X \times X) \setminus (F \times F)$.
\end{defn}

Note that all ghost operators in the Roe algebra form a closed $\ast$-ideal, denoted by $G(X;\H_0)$. Similarly, all ghost operators in the quasi-local algebra form a norm-closed $\ast$-ideal, denoted by $G_q(X;\H_0)$. When $\H_0 = \CC$, we simply write $G_u(X)$ and $G_{uq}(X)$ for $G(X;\CC)$ and $G_q(X;\CC)$, respectively.

\section{Proof of Theorem~\ref{thm:main result projection}}\label{sec:decomposition}

In this section, we prove Theorem~\ref{thm:main result projection}, which is divided into several parts.

\subsection{Analysis on a single space}\label{ssec:single piece}

Throughout this subsection, let $(X,d)$ be a discrete metric space, $\H_0$ a Hilbert space and $u^{(1)}, \cdots, u^{(M)}$ orthonormal in $\ell^2(X;\H_0)$. Denote by $U$ the subspace in $\ell^2(X;\H_0)$ spanned by $\{u^{(1)}, \cdots, u^{(M)}\}$, and $P_U$ the orthogonal projection onto $U$. Recall that $u_A:=\chi_A u$ for $u\in \ell^2(X;\H_0)$ and $A \subseteq X$.

When $U$ is one-dimensional and spanned by a unit vector $u$, \cite[Lemma 4.3]{LSZ23} shows the following:
\[
\|\chi_A P_U \chi_B\| = \|u_A\| \cdot \|u_B\| = \sqrt{\m(A) \cdot \m(B)} \quad \text{for} \quad A, B \subseteq X.
\]
Here $\m$ is the probability measure on $X$ given by $\m(\{x\}):=\|u(x)\|^2$ for $x\in X$.

For general $M$, we consider matrix-valued measures on $X$ as follows. Set $W: \CC^M \longrightarrow \ell^2(X;\H_0)$ by $W(e_i):=u^{(i)}$ for $i=1,2,\cdots, M$. Here \(e_i\) denotes the \(i\)-th standard basis vector of \(\CC^M\). Clearly,
\[
W^*W = \Id_{\CC^M} \quad \text{and} \quad WW^* = P_U.
\]
For $A \subseteq X$, denote
\[
\m(A):=W^* \chi_A W \in \M_M(\CC),
\]
where $\M_M(\CC)$ denotes the $M$-by-$M$ complex matrix algebra. If we want to emphasise the subspace $U$, we also denote $\m(A)$ by $\m_U(A)$.

It is clear that $0 \leq \m(A) \leq \I_M$, where $\I_M$ is the $M$-by-$M$ identity matrix. We record the following:
\begin{equation}\label{EQ:u and m}
\|u_A\|^2 = \langle \m(A) \xi, \xi \rangle = \|\m(A)^{1/2} \xi\|^2  \quad \text{for} \quad u=W \xi,
\end{equation}
which is useful later. Moreover, we have:

\begin{lem}\label{lem:norm of APB}
 With the same notation as above, we have:
 \[
   \|\chi_A P_U \chi_B\| = \left\|\m(A)^{1/2} \m(B)^{1/2}\right\| = \r(\m(A)\cdot \m(B))^{1/2}
 \]
 for any $A,B \subseteq X$.
 Here $\|\cdot\|$ denotes the operator norm and $\r(\cdot)$ denotes the spectral radius of the matrix. Hence, we have $\|\m(A)\m(B)\| \leq \|\chi_A P_U \chi_B\| \leq \|\m(A)\m(B)\|^{1/2}$.
\end{lem}

\begin{proof}
Letting $S:=\chi_A P_U \chi_B$, we have:
\[
 S^*S=\chi_BP_U\chi_AP_U\chi_B
     =\chi_BW\m(A)W^*\chi_B = CC^*,
\]
for $C:=\chi_BW\m(A)^{1/2}$. Hence
\begin{align*}
 \|S\|^2 &=\|S^*S\|=\|CC^*\|=\|C^*C\|=\left\|\m(A)^{1/2}\m(B)\m(A)^{1/2}\right\|\\
 &=\left\|\m(B)^{1/2}\m(A)^{1/2}\right\|^2 = \left\|\m(A)^{1/2}\m(B)^{1/2}\right\|^2,
\end{align*}
where the last equation holds since $\m(A)^{1/2}, \m(B)^{1/2}$ are self-adjoint.

The rest of the proof is clear from the following elementary fact: For $a,b \geq 0$ in a $C^*$-algebra $\A$, we have $\r(ab) = \|a^{1/2}b^{1/2}\|^2$. Moreover, if $0\leq a,b\leq 1$, we have $\|ab\| \leq \|a^{1/2}b^{1/2}\|$. The proof is standard and details are left to the reader.
\end{proof}

Now we show that bad quasi-local behaviour controls spectrum.

\begin{lem}\label{lem:forcing}
With the same notation as above, fix $\varepsilon \in (0,1)$ and set $\delta:=\varepsilon^2/16$. Assume that $P_U$ is $(\delta,R)$-quasi-local for some $R>0$.  If there exist a unit vector $u \in U$ and subsets $A, B \subseteq X$ with $d(A,B) > R$ such that $\|u_A\| \cdot \|u_B\|\geq \varepsilon$, then
\[
\lambda_{\min}(\m(A)) < \frac{\varepsilon^2}{4} \quad \text{and} \quad \lambda_{\max}(\m(A)) \geq \varepsilon^2.
\]
Here $\lambda_{\min}, \lambda_{\max}$ denote the minimal and maximal eigenvalue, respectively.
\end{lem}

\begin{proof}
Set $\xi:=W^*u$ for the isometry $W: \CC^M \longrightarrow \ell^2(X;\H_0)$ above. Then $\|\xi\| =1$ and according to \eqref{EQ:u and m}, we have $\langle \m(A) \xi, \xi\rangle = \|u_A\|^2 \geq \varepsilon^2$. So $\lambda_{\max}(\m(A))\geq \varepsilon^2$.

Now assume that $\m(A)\geq \frac{\varepsilon^2}{4} \cdot \Id_M$. Then $\m(A)$ is invertible and $\|\m(A)^{-1/2}\|\leq \frac{2}{\varepsilon}$. Using \eqref{EQ:u and m} and Lemma~\ref{lem:norm of APB}, we obtain:
\begin{align*}
\|u_B\| &=\|\m(B)^{1/2}\xi\|=\|\m(B)^{1/2}\m(A)^{1/2}\m(A)^{-1/2}\xi\|\\
&\le \|\m(B)^{1/2}\m(A)^{1/2}\| \cdot \|\m(A)^{-1/2}\|<
\frac{\varepsilon^2}{16}\cdot\frac{2}{\varepsilon}
=\frac{\varepsilon}{8}.
\end{align*}
This contradicts $\|u_B\| \geq \varepsilon$. Hence we conclude that $\lambda_{\min}(\m(A)) < \frac{\varepsilon^2}{4}$.
\end{proof}

\subsection{Decomposition}\label{ssec:proof of decomposition}

This subsection is devoted to the proof of Theorem~\ref{introthm:decomposition}.

For simplicity, we introduce the following:
For a metric space $X$, a projection $Q\in \B(\ell^2(X;\H_0))$ and parameters $\varepsilon \in (0,1)$, $R>0$, we say that $Q$ is \emph{$(\varepsilon,R)$-good} if for any unit vector $u\in\ran(Q)$ and any $A,B \subseteq X$ with $d(A,B) > R$, we have $\|u_A\| \cdot \|u_B\|<\varepsilon$. Otherwise, we say that $Q$ is \emph{$(\varepsilon,R)$-bad}.

Throughout this subsection, let $\{(X_n,d_n)\}_{n\in \NN}$ be a sequence of finite metric spaces of uniformly bounded geometry, and $X:=\bigsqcup_{n\in \NN} X_n$ be their coarse disjoint union. Let $M \in \NN$, and $\H_0$ be a Hilbert space.

We start with the following global reduction, which is crucial.

\begin{prop}\label{prop:global reduction}
Fix $\varepsilon\in (0,1)$ and set $\delta:=\varepsilon^2/16$. Let $Q=\SOT$-$\sum_{n\in \NN} Q_n$ be a block-diagonal projection in $C^*_q(X;\H_0)$ with $1\leq \rank(Q_n) \leq M$ for each $n$. Take $R>0$ so that $Q$ is $(\delta,R)$-quasi-local. Then for each $n$, there is a non-zero projection $Q'_n \leq Q_n$ such that
\[
\begin{cases}
Q_n'=Q_n, &\text{if }Q_n\text{ is }(\varepsilon,R)\text{-good};\\
\rank(Q_n')<\rank(Q_n),&\text{if }Q_n\text{ is }(\varepsilon,R)\text{-bad},
\end{cases}
\]
and the block-diagonal projection $Q':=\SOT$-$\sum_{n\in \NN} Q'_n \in C^*_q(X;\H_0)$.
\end{prop}

\begin{proof}
When $\rank(Q_n) =1$, the bad case cannot happen and there is nothing to change. Hence without loss of generality, we assume $\rank(Q_n) >1$ for each $n\in \NN$.

Denote $\Gamma:=\{n\in \NN: Q_n \text{ is }(\varepsilon,R)\text{-bad}\}$.
For each $n\in \Gamma$, choose a unit vector $u_n\in\ran(Q_n)$ and $A_n,B_n\subseteq X_n$ witnessing badness.  By Lemma~\ref{lem:forcing}, the positive matrix $T_n:=\m_{\ran(Q_n)}(A_n)$ has at least one eigenvalue in $[0,\frac{\varepsilon^2}{4})$ and one in $[\varepsilon^2,1]$.

Choose $M-1$ pairwise disjoint closed intervals $I_1,\cdots,I_{M-1}$ in $\left(\frac{\varepsilon^2}{4},\varepsilon^2\right)$ with non-empty interiors. For each $n \in \Gamma$, since $T_n$ has at most $M$ distinct eigenvalues and at least two of them sit outside $\left(\frac{\varepsilon^2}{4},\varepsilon^2\right)$, there is $j(n) \in \{1,\cdots,M-1\}$ such that $I_{j(n)}$ is disjoint from $\sigma(T_n)$. For each $j\in \{1,\cdots,M-1\}$, set
\[
\Gamma_j:=\{n \in \Gamma: j(n)=j\}, \quad A^{(j)}:=\bigsqcup_{n\in \Gamma_j} A_n \quad \text{and} \quad Q^{(j)}:=\SOT\text{-}\sum_{n\in \Gamma_j} Q_n.
\]
It is clear that $Q^{(j)} \in C^*_{q}(X;\H_0)$. Set
\[
S^{(j)}:=Q^{(j)}\chi_{A^{(j)}}Q^{(j)} = \SOT\text{-}\sum_{n\in \Gamma_j} Q_n\chi_{A_n}Q_n \in C^*_{q}(X;\H_0).
\]
Note that for each $n\in \Gamma_j$, the restriction of $S^{(j)}_n:=Q_n\chi_{A_n}Q_n$ on $\ran(Q_n)$ is unitarily equivalent to $T_n$.

Writing $I_j=[a_j,b_j]$, we have $\sigma(S^{(j)}) \cap (a_j,b_j) = \emptyset$. Choose a continuous function $f_j:[0,1]\to[0,1]$ such that $f_j = 0$ on $[0,a_j]$ and $f_j=1$ on $[b_j,1]$. Then $F^{(j)}:=f_j(S^{(j)})$ is a projection in $C^*_{q}(X;\H_0)$ and $F^{(j)} \leq Q^{(j)}$. Write $F^{(j)}= \SOT\text{-}\sum_{n\in \Gamma_j} F^{(j)}_n$. For each $n\in \Gamma_j$, $S^{(j)}_n$ has an eigenvalue in $[0,a_j]$ and an eigenvalue in $[b_j,1]$ and hence,
\[
0<\rank(F^{(j)}_n)<\rank(Q_n).
\]
Therefore, we set
\[
Q':=\sum_{j=1}^{M-1} F^{(j)} + \SOT\text{-}\sum_{n\in \NN \setminus \Gamma} Q_n,
\]
which fulfills the task.
\end{proof}

Iterating the global reduction several times, we have:

\begin{cor}\label{cor:one step reduction}
Let $Q=\SOT$-$\sum_{n\in \NN} Q_n$ be a block-diagonal projection in $C^*_q(X;\H_0)$ with $1\leq \rank(Q_n) \leq M$ for each $n$. Then for every $\varepsilon \in (0,1)$, there are non-zero projections $\widetilde Q_n \leq Q_n$ such that the block-diagonal projection
\[
 \widetilde{Q}:=\SOT\text{-}\sum_{n\in \NN}\widetilde Q_n
\]
is in $C^*_q(X;\H_0)$ and there is $S>0$ for which every $\widetilde Q_n$ is $(\varepsilon,S)$-good.
\end{cor}

\begin{proof}
If \(M=1\), simply take \(\widetilde Q=Q\). Hence in the following, we assume \(M\geq2\).

Set $\delta:=\frac{\varepsilon^2}{16}$ and we start with $Q^{(0)}:=Q$. Since $Q^{(0)}\in C_q^*(X;\H_0)$, choose $R_0>0$ such that $Q^{(0)}$ is $(\delta,R_0)$-quasi-local. Applying Proposition~\ref{prop:global reduction} to $Q^{(0)}$ and $\varepsilon, \delta, R_0$, we obtain $Q^{(1)} = \SOT\text{-}\sum_{n\in \NN} Q^{(1)}_n \leq Q^{(0)}$ in $C^*_q(X;\H_0)$ such that for $(\varepsilon, R_0)$-good $Q_n$ we have $Q^{(1)}_n = Q_n$, and for $(\varepsilon, R_0)$-bad $Q_n$ we have $0<\rank(Q^{(1)}_n) < \rank(Q_n)$. Now choose $R_1>R_0$ such that $Q^{(1)}$ is $(\delta,R_1)$-quasi-local. Applying Proposition~\ref{prop:global reduction} to $Q^{(1)}$ and $\varepsilon, \delta, R_1$, we obtain $Q^{(2)}\in C^*_q(X;\H_0)$ satisfying a similar condition.

Repeating this procedure $M-1$ times, we obtain a sequence of decreasing projections
\[
Q=Q^{(0)} \geq Q^{(1)} \geq \cdots \geq Q^{(M-1)} \quad \text{in} \quad C^*_q(X;\H_0)
\]
and $R_0 < R_1 < \cdots < R_{M-2}$.
If a block $Q_n^{(r)}$ is $(\varepsilon,R_r)$-good, then it is unchanged at that round.  Since $\{R_r\}_r$ is increasing, it remains good at every later stage.  If it is $(\varepsilon,R_r)$-bad, its rank strictly decreases.

Finally, choose $S > R_{M-2}$ such that $Q^{(M-1)}$ is $(\varepsilon,S)$-quasi-local. If a block $Q^{(r)}_n$ is $(\varepsilon,R_r)$-good for some $r$, then it remains unchanged and $Q^{(M-1)}_n = Q^{(r)}_n$, which is also $(\varepsilon,S)$-good. Otherwise, the block $Q^{(r)}_n$ is $(\varepsilon,R_r)$-bad for all $r$ and hence, the rank strictly decreases each time. This shows that $\rank(Q^{(M-1)}_n) = 1$ and then, $Q^{(M-1)}_n$ is also $(\varepsilon,S)$-good. Setting $\widetilde Q:= Q^{(M-1)}$, we conclude the proof.
\end{proof}

Now we are ready to prove the case of projections in Theorem~\ref{introthm:decomposition}. Let us state it in the following precise form:

\begin{thm}\label{thm:decomposition of projections}
 Let $\{(X_n,d_n)\}_{n\in \NN}$ be a sequence of finite metric spaces of uniformly bounded geometry, and $X:=\bigsqcup_{n\in \NN} X_n$ be their coarse disjoint union. Let $M \in \NN$ and $\H_0$ be a Hilbert space. Let $P=\SOT$-$\sum_{n\in \NN} P_n$ be a block-diagonal projection in $C^*_q(X;\H_0)$ with $1 \leq \rank(P_n) \leq M$ for each $n$. Then there is a block-rank-one projection $Q=\SOT$-$\sum_{n\in \NN} Q_n$ in $C^*_q(X;\H_0)$ such that $Q \leq P$.

Consequently, there are mutually orthogonal block-diagonal projections $P^{(1)}, \cdots, P^{(M)}$ in $C^*_q(X;\H_0)$ with block-rank at most one such that $P=P^{(1)}+\cdots+P^{(M)}$.
\end{thm}

\begin{proof}
For each $k\in \NN$, let $\varepsilon_k:=2^{-k}$. Starting from $P^{(0)}=P$, apply Corollary~\ref{cor:one step reduction} successively to obtain a sequence of block-diagonal projections
\[
 P=P^{(0)}\geq P^{(1)}\geq P^{(2)}\geq\cdots \quad \text{in} \quad C^*_q(X;\H_0)
\]
such that every block $P_n^{(k)}$ is non-zero and $(\varepsilon_k,S_k)$-good for some $S_k>0$. Here we write $P^{(k)}=\SOT\text{-}\sum_{n\in \NN} P_n^{(k)}$ for each $k\in \NN$.

For each $n\in \NN$, the nested sequence
\[
 \ran(P_n^{(0)})\supseteq\ran(P_n^{(1)})\supseteq\ran(P_n^{(2)})\supseteq\cdots
\]
consists of non-zero subspaces of a space of dimension at most $M$.  Hence it eventually stabilises, and therefore
\[
 V_n^\infty:=\bigcap_{k\geq 1}\ran(P_n^{(k)})\neq \{0\}.
\]
Choose a unit vector $u_n\in V_n^\infty$ and put $Q_n:=P_{u_n}$. Set $Q:=\SOT\text{-}\sum_{n\in \NN} Q_n$.

To show $Q$ is quasi-local, fix $\eta>0$ and choose $k$ so that $2^{-k}<\eta$. For each $n$ and $A,B \subseteq X_n$ with $d(A,B)>S_k$, since $u_n\in\ran(P_n^{(k)})$ and $P_n^{(k)}$ is $(2^{-k},S_k)$-good, we have
\[
 \|\chi_A Q_n \chi_B\| = \|(u_n)_{A}\|\cdot\|(u_n)_{B}\|<2^{-k} < \eta.
\]
Thus $Q$ is quasi-local. On the other hand, each block $Q_n$ having rank $1$ implies that $Q$ is locally compact. Thus, $Q \in C^*_q(X;\H_0)$.

The last sentence is clear by doing the extraction above $M$ times.
\end{proof}

\begin{cor}\label{cor:decomposition pi}
Let $V=\SOT$-$\sum_{n\in \NN} V_n$ be a block-diagonal partial isometry in $C^*_q(X;\H_0)$ with block-rank at most $M$. Then there are block-diagonal partial isometries $V^{(1)}, \cdots, V^{(M)}$ in $C^*_q(X;\H_0)$ with block-rank at most one such that $V=V^{(1)}+\cdots+V^{(M)}$.
\end{cor}

\begin{proof}
Set $P=V^*V$, which is a projection in $C^*_q(X;\H_0)$. Applying Theorem~\ref{thm:decomposition of projections} to $P$, we obtain a decomposition $P=P^{(1)} + \cdots + P^{(M)}$, where each block-diagonal operator $P^{(i)}\in C^*_q(X;\H_0)$ has block-rank at most one. Then
\[
V=VP = VP^{(1)} + \cdots + VP^{(M)}.
\]
Each $VP^{(i)}$ is a partial isometry in $C^*_q(X;\H_0)$ with block-rank at most one.
\end{proof}

Finally, we deal with the general case of Theorem~\ref{introthm:decomposition}. We use a spectral method similar to that used in the proof of Proposition~\ref{prop:global reduction}.

\begin{thm}\label{thm:peeling}
Let $\{(X_n,d_n)\}_{n\in \NN}$ be a sequence of finite metric spaces of uniformly bounded geometry, and $X:=\bigsqcup_{n\in \NN} X_n$ be their coarse disjoint union. Let $M \in \NN$ and $\H_0$ be a Hilbert space. Let $T=\SOT$-$\sum_{n\in \NN} T_n$ be a block-diagonal operator in $C^*_q(X;\H_0)$ with block-rank at most $M$. Then there exists a block-diagonal operator $S=\SOT$-$\sum_{n\in \NN} S_n$ in $C^*_q(X;\H_0)$ with block-rank at most one such that for every $n$, we have
\[
\rank(T_n-S_n)=\max\left\{\rank(T_n)-1,0\right\}.
\]
Consequently, there are block-diagonal operators $T^{(1)}, \cdots, T^{(M)}$ in $C^*_q(X;\H_0)$ with block-rank at most one such that $T=T^{(1)}+\cdots+T^{(M)}$.
\end{thm}

\begin{proof}
After rescaling and discarding the zero blocks, we may assume $0<\|T_n\|\le1$ for each $n$. For $k \in \NN \cup \{0\}$, set
\[
\Gamma_k:=\{n:2^{-k-1}<\|T_n\|\le2^{-k}\}
\]
and choose $M$ pairwise disjoint closed intervals $I_{k,1},\ldots,I_{k,M}$ in $(2^{-k-2},2^{-k-1})$ with non-empty interiors. For $n\in\Gamma_k$, $|T_n|$ has at most $M$ non-zero eigenvalues and the largest one equals $\|T_n\|$, which is greater than $2^{-k-1}$. Hence, we choose $j(n) \in \{1,\cdots,M\}$ such that $I_{k,j(n)} \cap \sigma(|T_n|) = \emptyset$. For each $k$ and $j\in \{1,\cdots,M\}$, we set:
\[
\Gamma_{k,j}:=\{n\in\Gamma_k:j(n)=j\}
\quad \text{and} \quad
T^{(k,j)}:= \SOT\text{-}\sum_{n\in \Gamma_{k,j}} T_n.
\]
Write $I_{k,j}=[a_{k,j},b_{k,j}]$ and choose a continuous function $f_{k,j}:[0,1]\to[0,1]$ such that $f_{k,j} = 0$ on $[0,a_{k,j}]$ and $f_{k,j} = 1$ on $[b_{k,j},1]$. Then
\[
P^{(k,j)}:=f_{k,j}(|T^{(k,j)}|) = \SOT\text{-}\sum_{n\in \Gamma_{k,j}} f_{k,j}(|T_n|)
\]
is a projection in $C_q^*(X;\H_0)$. Writing $P^{(k,j)} = \SOT\text{-}\sum_{n\in \Gamma_{k,j}} P^{(k,j)}_n$, we have $P^{(k,j)}_n \neq 0$ with range contained in $\ran(|T_n|)$. Hence $P^{(k,j)}$ has block-rank at most $M$.

Applying Theorem~\ref{thm:decomposition of projections} to $P^{(k,j)}$, we obtain a block-diagonal projection $Q^{(k,j)} = \SOT\text{-}\sum_{n\in \Gamma_{k,j}} Q^{(k,j)}_n$ in $C^*_q(X;\H_0)$ such that $\rank(Q^{(k,j)}_n) =1$ and $Q^{(k,j)}_n \leq P^{(k,j)}_n$ for each $n\in \Gamma_{k,j}$. Set
\[
S^{(k,j)}:=T^{(k,j)}Q^{(k,j)} \quad \text{and} \quad S^{(k)}:=\sum_{j=1}^M S^{(k,j)}.
\]
Then $S^{(k)} \in C^*_q(X;\H_0)$ and we write $S^{(k)} = \SOT\text{-}\sum_{n\in \Gamma_k} S^{(k)}_n$. For each $n\in \Gamma_k$, note that $S^{(k)}_n = T_nQ^{(k,j(n))}_n$ and
\[
\ran(Q^{(k,j(n))}_n) \subseteq \ran(P^{(k,j(n))}_n) \subseteq \ran(|T_n|) = \ker(T_n)^\perp.
\]
Hence $\rank(S^{(k)}_n) =1$ and
\[
\rank(T_n - S^{(k)}_n) = \rank(T_n(\Id - Q^{(k,j(n))}_n)) = \rank(T_n) - 1.
\]
On the other hand, we have $\|S^{(k)}_n\| \leq \|T_n\| \leq 2^{-k}$ and hence, $\|S^{(k)}\| \leq 2^{-k}$.

Finally, we set
\[
S:=\sum_{k=0}^\infty S^{(k)}.
\]
Since $\Gamma_k$ are mutually disjoint, we have
\[
\left\| \sum_{k>K} S^{(k)} \right\| = \sup_{k>K} \|S^{(k)}\| \leq 2^{-K} \quad \text{for each} \quad K\in \NN.
\]
Hence the series $\sum_{k=0}^\infty S^{(k)}$ converges in norm, which shows that $S \in C^*_q(X;\H_0)$. Therefore, $S$ satisfies all the conditions required and we finish the proof.
\end{proof}

\subsection{The block-rank-one case}

Based on Theorem~\ref{thm:peeling}, the proof of Theorem~\ref{thm:main result projection} is now reduced to the block-rank-one case. Note that block-rank-one projections have already been studied in \cite[Theorem 6.6]{LSZ23} and here, we focus on the general case. We use the same notation as in the previous subsection.

Let us start with the case of partial isometries.

\begin{lem}\label{lem:rankone-pi}
Let $V=\SOT$-$\sum_{n\in \NN} V_n$ be a block-diagonal partial isometry in $C^*_q(X;\H_0)$ with block-rank at most one. Then $V \in C^*(X;\H_0)$.
\end{lem}

\begin{proof}
Let $P:=V^*V$ and $Q:=VV^*$. Define an operator $G$ on $\ell^2(X;\H_0\oplus\H_0)$ by
\[
G:=\frac12
\begin{pmatrix}
P&V^*\\
V&Q
\end{pmatrix}.
\]
A direct calculation shows that $G$ is a projection in $C^*_q(X;\H_0 \oplus \H_0)$. Moreover, if $V_n = 0$, then $G_n=0$. Otherwise, choose unit vectors $\xi_n,\eta_n$ with $V_n\xi_n=\eta_n$ and we have $G_n=P_{\frac1{\sqrt2}(\xi_n,\eta_n)}$.
Therefore, $G$ has block-rank at most one. By \cite[Theorem 6.6]{LSZ23} (where $\H_0$ is assumed to be separable, but the proof can be applied to the general case), $G \in C^*(X;\H_0 \oplus \H_0)$. Taking the lower left corner gives $V \in C^*(X;\H_0)$.
\end{proof}

Now we consider general block-rank-one operators.

\begin{prop}\label{prop:rankone-op}
Let $T=\SOT$-$\sum_{n\in \NN} T_n$ be a block-diagonal operator in $C^*_q(X;\H_0)$ with block-rank at most one. Then $T \in C^*(X;\H_0)$.
\end{prop}

\begin{proof}
Given $\varepsilon>0$, we set
\[
\Gamma_\varepsilon:=\{n \in \NN:\|T_n\|\ge\varepsilon\}
\quad \text{and} \quad
T_\varepsilon:=\SOT\text{-}\sum_{n\in \Gamma_\varepsilon} T_n.
\]
Then $T_\varepsilon\in C_q^*(X;\H_0)$ and $\|T-T_\varepsilon\| \leq \varepsilon$. For $n\in\Gamma_\varepsilon$, write $T_n=\lambda_nV_n$, where $\lambda_n:=\|T_n\|$ and $V_n$ is a rank-one partial isometry. Since $|\lambda_n^{-1}| \leq \varepsilon^{-1}$, the operator
\[
V_\varepsilon:=\SOT\text{-}\sum_{n\in \Gamma_\varepsilon} V_n = \SOT\text{-}\sum_{n\in \Gamma_\varepsilon} \lambda_n^{-1} T_n
\]
belongs to $C^*_q(X;\H_0)$. Hence, Lemma~\ref{lem:rankone-pi} shows that $V_\varepsilon \in C^*(X;\H_0)$, which implies that $T_\varepsilon\in C^*(X;\H_0)$. Letting $\varepsilon \to 0+$, we obtain that $T \in C^*(X;\H_0)$.
\end{proof}

Finally, we are ready to prove Theorem~\ref{thm:main result projection}.

\begin{proof}[Proof of Theorem~\ref{thm:main result projection}]
Due to Theorem~\ref{thm:peeling}, we can find block-diagonal operators $T^{(1)}, \cdots, T^{(M)}$ in $C^*_q(X;\H_0)$ with block-rank at most one such that $T=T^{(1)}+\cdots+T^{(M)}$. For each $i=1, \cdots, M$, Proposition~\ref{prop:rankone-op} shows that $T^{(i)} \in C^*(X;\H_0)$. Hence we conclude that $T \in C^*(X;\H_0)$.
\end{proof}

\section{The case of unbounded block-rank}\label{sec:unbounded case}

In this section, we prove Theorem~\ref{prop:infiniteprojection intro}, based on Ozawa's recent work \cite{Oza25}.

Recall that a sequence of finite metric spaces $\{X_n\}_{n\in \NN}$ is a sequence of \emph{asymptotic expanders} if they have uniformly bounded geometry, $\lim_{n\to \infty}|X_n| = \infty$ and for any $\alpha>0$, there exist $c_\alpha, R_\alpha>0$ such that for any $n\in \NN$ and $A \subseteq X_n$ with $\alpha |X_n| \leq |A| \leq \frac{1}{2}|X_n|$, we have $|\partial_{R_\alpha} A| > c_\alpha |A|$. This is the notion introduced in \cite{LNSZ21}, with uniformly bounded geometry imposed as a standing assumption. When \(\{X_n\}_n\) is a sequence of graphs and one can take \(c_\alpha=c>0\) and \(R_\alpha=1\) independently of \(\alpha\), this recovers the classical notion of an expander sequence.

We need the following stronger form of \cite[Theorem B]{Oza25}:

\begin{prop}\label{prop:Ozawa's Theorem B}
Let $(X,d)$ be a discrete metric space of bounded geometry which contains a sequence of asymptotic expanders. Then the direct product of matrix algebras $\Pi_{n\in \NN} \M_n(\CC)$ can be embedded into $G_{uq}(X)$, the ideal of ghost operators in $C^*_{uq}(X)$.
\end{prop}

\begin{proof}
  We follow the proof of \cite[Theorem B]{Oza25}. Assume that $X$ contains a sequence $\{X_n\}_{n\in \NN}$ of asymptotic expanders. For each $k\in \NN$ with $k \geq 2$, set $\delta_k := 1/k$ and $\varepsilon_k := 100\sqrt{\delta_k \log(1/\delta_k)}$. By \cite[Lemma 4]{Oza25}, for each $n\in \NN$ we can find an $n$-dimensional subspace $V_{\RR}(n) \subseteq \RR^{d(n)}$ for some $d(n) \geq n$ such that
  \[
    \max \left\{ \left\| P_{V_{\RR}(n)}|_{\ell^2(E;\RR)} \right\| : E \subset \{1,\cdots, d(n)\} \text{~and~}  |E|/d(n)\leq \delta_k \right\} < \varepsilon_k
  \]
  for all $k = 2,\cdots,n$. Set $V(n):=V_{\RR}(n) \otimes \CC$ for each $n\in \NN$.

  Taking a subsequence if necessary, we assume that $|X_n| \geq d(n)$ and then embed $V(n)$ into $\ell^2(X_n)$. This provides an embedding of $\prod_{n\in \NN} \B(V(n))$ into $\B(\ell^2(X))$. Following the proof of \cite[Theorem B]{Oza25}, this embedding has image in the quasi-local algebra $C^*_{uq}(X)$. It remains to show that its image consists of ghost operators.

  To see this, take an arbitrary $u = (u_n)_{n\in \NN} \in \prod_n \B(V(n))$ with norm $1$. Given $\varepsilon>0$, take $k\geq 2$ with $\varepsilon_k < \varepsilon$ and $N > k$ such that $1/d(n) < \delta_k$ for $n > N$. Then for any $x,y\in X_n$ with $n>N$, we have
  \begin{align*}
  |(u_n)_{x,y}| &\leq \|u_n\chi_{\{y\}}\| = \|u_n P_{V(n)}\chi_{\{y\}} \| \leq \|P_{V(n)}\chi_{\{y\}}\| = \|P_{V(n)}|_{\ell^2(\{y\})}\| = \|P_{V_{\RR}(n)}|_{\ell^2(\{y\};\RR)}\| \\
  & < \varepsilon_k < \varepsilon.
  \end{align*}
  Hence, $u$ is a ghost operator and we finish the proof.
\end{proof}

Now we prove Theorem~\ref{prop:infiniteprojection intro}.
\begin{proof}[Proof of Theorem~\ref{prop:infiniteprojection intro}]
Applying \cite[Theorem A]{Oza25} and Proposition~\ref{prop:Ozawa's Theorem B}, we assume that (after taking a subsequence if necessary) there exists a linear subspace $V(n) \subseteq \ell^2(X_n)$ such that
  \[
  \prod_n \B(V(n)) \subseteq G_{uq}(X) \quad \text{while} \quad \prod_n \B(V(n)) \nsubseteq C_{u}^*(X).
  \]
  Take a positive contractive element $A=\SOT$-$\sum_{n\in \NN} A_n \in \prod_n \B(V(n))$ while $A \notin C^*_u(X)$. Hence the spectrum of each $A_n$ is contained in $[0,1]$.

  Define $\varphi:[0,1] \to [0,1]$ by $\varphi(x) = x$ for $x\in [0,1]$. For each $k\in \NN$, define
  \[
   \varphi^{(k)}: [0,1] \longrightarrow [0,1] \quad \text{by} \quad x \mapsto \sum_{i=1}^{k} \frac{i}{k} \cdot \chi_{\left(\frac{i-1}{k}, \frac{i}{k}\right]}(x).
  \]
  It is clear that $\|\varphi^{(k)} - \varphi\|_{\infty} \leq \frac{1}{k}$. Since $A_n \in \B(V(n))$ has discrete spectrum, $\varphi^{(k)}(A_n) \in \B(V(n))$ is well-defined for each $k,n\in \NN$. Moreover, we have
  \[
    \left\|\left(\SOT\text{-}\sum_{n\in \NN} \varphi^{(k)}(A_n)\right) - A\right\| = \sup_{n \in \NN} \|\varphi^{(k)}(A_n) - A_n\| \leq \|\varphi^{(k)} - \varphi\|_{\infty} \leq \frac{1}{k}.
  \]
  Since $A \notin C^*_u(X)$, there exists $k_0 \in \NN$ such that
  \begin{equation}\label{EQ:A notin Roe}
    \SOT\text{-}\sum_{n\in \NN} \varphi^{(k_0)}(A_n) \notin C_{u}^*(X).
  \end{equation}

  On the other hand, we have
  \[
    \SOT\text{-}\sum_{n\in \NN} \varphi^{(k_0)}(A_n) = \SOT\text{-}\sum_{n\in \NN} \left(\sum_{i=1}^{k_0} \frac{i}{k_0} \chi_{\left(\frac{i-1}{k_0}, \frac{i}{k_0}\right]} (A_n) \right) = \sum_{i=1}^{k_0} \frac{i}{k_0} \left(\SOT\text{-}\sum_{n\in \NN} \chi_{\left(\frac{i-1}{k_0}, \frac{i}{k_0}\right]}(A_n)\right).
  \]
  Hence by (\ref{EQ:A notin Roe}), there exists $i_0 \in \{1,\cdots,k_0\}$ such that
  \[
   P:=\SOT\text{-}\sum_{n\in \NN} \chi_{\left(\frac{i_0-1}{k_0}, \frac{i_0}{k_0}\right]}(A_n) \notin C^*_u(X).
  \]
  Note that $P$ is a projection, which belongs to $\prod_{n\in \NN}\B(V(n)) \subseteq G_{uq}(X)$. Therefore, we conclude the proof.
\end{proof}

Finally, we show that the same phenomenon happens in the coefficient case:

\begin{cor}\label{cor:infiniteprojection Roe case}
 Let $\{X_n\}_{n \in \mathbb{N}}$ be a sequence of asymptotic expanders, $X = \bigsqcup_{n=1}^{\infty} X_n$ be their coarse disjoint union and $\H_0$ be a non-zero Hilbert space. Then there exists a ghost block-diagonal projection $Q = \SOT$-$\sum_{n\in \NN} Q_n$ in $C_{q}^*(X;\H_0)$ while $Q \notin C^*(X;\H_0)$.
\end{cor}

\begin{proof}
Take $P:=\SOT$-$\sum_{n\in \NN} P_n$ from Theorem~\ref{prop:infiniteprojection intro}, and fix a rank-one projection $e \in \B(\H_0)$. For each $n\in \NN$, set $Q_n:=P_n \otimes e$ and define $Q:= \SOT$-$\sum_{n\in \NN} Q_n = P\otimes e$. Since $P \in G_{uq}(X)$, it follows directly that $Q \in C^*_q(X;\H_0)$ and $Q$ is a ghost projection.

  On the other hand, if $Q \in C^*(X;\H_0)$, there would exist locally compact finite propagation operators $T^{(k)}=\SOT\text{-}\sum_{n\in \NN}T^{(k)}_n \in \B(\ell^2(X;\H_0))$ with $\|Q - T^{(k)}\| \to 0$ as $k \to \infty$.
  Note that $(\Id_{\ell^2(X)} \otimes e) Q (\Id_{\ell^2(X)} \otimes e) = Q$ and
  \[
   (\Id_{\ell^2(X)} \otimes e) T^{(k)} (\Id_{\ell^2(X)} \otimes e) = \SOT\text{-}\sum_{n\in \NN} (\Id_{\ell^2(X_n)}\otimes e)T^{(k)}_n(\Id_{\ell^2(X_n)}\otimes e).
  \]
  Since $(\Id_{\ell^2(X_n)}\otimes e)T^{(k)}_n(\Id_{\ell^2(X_n)}\otimes e)$ can be written as $S^{(k)}_n \otimes e$ for some $S^{(k)}_n \in \B(\ell^2(X_n))$ with propagation at most $\ppg(T^{(k)}_n)$, the operator $S^{(k)}:=\SOT$-$\sum_{n\in \NN}S^{(k)}_n$ satisfies
  \begin{align*}
  \|P-S^{(k)}\| &= \|Q - (S^{(k)} \otimes e)\|\\
  &= \left\|(\Id_{\ell^2(X)} \otimes e) Q (\Id_{\ell^2(X)} \otimes e) - \left( \SOT\text{-}\sum_{n\in \NN} (\Id_{\ell^2(X_n)}\otimes e)T^{(k)}_n(\Id_{\ell^2(X_n)}\otimes e) \right)\right\|\\
  & \leq \left\|Q - \left(\SOT\text{-}\sum_{n\in \NN}T^{(k)}_n\right)\right\| = \|Q - T^{(k)}\| \to 0 \quad \text{as} \quad k\to \infty,
  \end{align*}
  which contradicts $P \notin C^*_u(X)$ from Theorem~\ref{prop:infiniteprojection intro}.
\end{proof}

\section{The quotient isomorphism result}\label{sec:quotient isom}

The aim of this section is to prove Theorem~\ref{thm:quotient isom intro}.
Let us start with some preparations on the notion of limit operators.

\subsection{Limit operators}\label{ssec:limit operators}

Here we recall the notion of limit spaces and limit operators from \cite{GQW24, SW17}, and generalise the latter to quasi-local operators with coefficient of compact operators.

Let $(X,d)$ be a strongly discrete (\emph{i.e.}, $\{d(x,y): x,y\in X\}$ is discrete in $\RR$) metric space of bounded geometry, and $\H_0$ be a Hilbert space. Clearly, every graph with the edge-path metric is strongly discrete.

Given $\omega \in \partial \beta X$, we say that a partial translation $t: D \to R$ for $D,R \subseteq X$ is \emph{compatible with $\omega$} if $D\in\omega$.  Its continuous extension sends $\omega$ to an ultrafilter, denoted by $t(\omega)$. The \emph{limit space at $\omega$} is
\[
 X(\omega):=\{t(\omega):t\text{ is a partial translation compatible with }\omega\}.
\]
Choose a compatible family $\{t_\alpha:D_\alpha\to R_\alpha\}_{\alpha\in X(\omega)}$ satisfying $t_\alpha(\omega)=\alpha$. It is shown in \cite[Proposition~3.7]{SW17} that the formula
\begin{equation}\label{EQ:limit metric}
 d_\omega(\alpha,\beta):=\lim_{x\to\omega}d(t_\alpha(x),t_\beta(x))
\end{equation}
defines a strongly discrete metric $d_\omega$ on $X(\omega)$, which also has bounded geometry.

For $T \in C^*(X;\H_0)$ and $\omega \in \partial \beta X$, a limit operator $\Phi_\omega(T)$ is defined in \cite{SW17} when $T$ is \emph{rich} at $\omega$ (see \cite[Definition 4.1]{SW17} for a precise definition). When $\H_0$ is finite-dimensional, every operator in $C^*(X;\H_0)$ is rich. However, this is not the case in general. To overcome the issue, we modify the definition using ultrapowers.

Given $\omega \in \partial \beta X$, denote the ultrapower
\[
\K_\omega:=(\H_0)_\omega = \ell^\infty(X;\H_0) / \N_\omega, \quad \text{where} \quad \N_\omega:=\left\{(\xi_x)_x:\lim_{x\to\omega}\|\xi_x\|=0\right\}.
\]
For $(\xi_x)_x \in \ell^\infty(X;\H_0)$, denote its image in $\K_\omega$ by $[(\xi_x)]_\omega$. Note that $\K_\omega$ forms a Hilbert space with the inner product
\[
\left\langle [(\xi_x)]_\omega, [(\eta_x)]_\omega \right\rangle:= \lim_{x\to \omega} \langle \xi_x, \eta_x \rangle.
\]

For $(A_x)_x \in \ell^\infty(X;\B(\H_0))$, it induces an operator
\[
[(A_x)]_\omega: \K_\omega \to \K_\omega \quad \text{by} \quad [(\xi_x)]_\omega \mapsto [(A_x\xi_x)]_\omega.
\]
It is clear that $[(A_x)]_\omega$ is a bounded linear operator on $\K_\omega$ with norm
\begin{equation}\label{EQ:norm compute ultra}
\|[(A_x)]_\omega\| = \lim_{x\to\omega}\|A_x\|.
\end{equation}
Hence we obtain a $C^*$-homomorphism
\begin{equation}\label{EQ:map q}
q_\omega: \ell^\infty(X;\B(\H_0)) \to \B(\K_\omega) \quad \text{by} \quad (A_x)_x \mapsto [(A_x)]_\omega.
\end{equation}

For $(A_x)_{x\in D} \in \ell^\infty(D;\B(\H_0))$ with $D \in \omega$, we can extend it to some $(A_x)_{x\in X} \in \ell^\infty(X;\B(\H_0))$ (\emph{e.g.}, the zero extension). It is routine to check that the element $q_\omega((A_x)_{x\in X})$ is independent of the extension. Hence in the sequel, we will use the same notation $[(A_x)]_\omega$ for such $(A_x)_{x\in D}$  without further explanation.

For $T \in \B(\ell^2(X;\H_0))$ and $\omega \in \partial \beta X$, define the \emph{limit operator} $\Phi_\omega(T)$ by:
\begin{equation}\label{EQ:limit operator general}
\left(\Phi_\omega(T)\right)_{\alpha,\beta}
 :=q_\omega\left((T_{t_\alpha(x),t_\beta(x)})_x\right) = \big[\big(T_{t_\alpha(x),t_\beta(x)}\big)\big]_\omega \in\B(\K_\omega) \quad \text{for} \quad \alpha, \beta \in X(\omega).
\end{equation}
We have the following:

\begin{lem}\label{lem:matrix boundedness}
For $T \in \B(\ell^2(X;\H_0))$ and $\omega \in \partial \beta X$, the matrix in \eqref{EQ:limit operator general} defines an operator
\[
 \Phi_\omega(T) \in \B\left(\ell^2(X(\omega); \K_\omega)\right)
\]
with $\|\Phi_\omega(T)\|\leq\|T\|$ and we have
\begin{equation}\label{EQ:norm compute limit op}
 \left\|\left(\Phi_\omega(T)\right)_{\alpha,\beta}\right\|
 =\lim_{x\to\omega}\|T_{t_\alpha(x),t_\beta(x)}\|.
\end{equation}
Moreover, if $T$ is $(\varepsilon,R)$-quasi-local for some $\varepsilon,R>0$, then $\Phi_\omega(T)$ is $(2\varepsilon,R)$-quasi-local.
\end{lem}

\begin{proof}
Let $F \subseteq X(\omega)$ be finite. By \cite[Proposition 3.10]{SW17}, there exists $Y \in \omega$ such that for each $x\in Y$, the map $\alpha\mapsto t_\alpha(x)$ is an isometry on $F$. For $x\in Y$, denote
\[
 U_x:\ell^2(F;\H_0)\longrightarrow\ell^2(X;\H_0),
 \qquad
\delta_\alpha\otimes\xi \mapsto \delta_{t_\alpha(x)}\otimes\xi.
\]
Clearly, we have $(U_x^*TU_x)_{\alpha,\beta} = T_{t_\alpha(x), t_\beta(x)}$ for any $\alpha, \beta \in F$. Since the map $q_\omega$ in \eqref{EQ:map q} is completely contractive and $\|U_x^*TU_x\|\leq\|T\|$ for each $x\in Y$, we obtain
\[
\|\chi_F \Phi_\omega(T) \chi_F\| \leq \sup_{x\in Y} \|U_x^*TU_x\| \leq \|T\|.
\]
Hence $\Phi_\omega(T)$ is bounded with norm at most $\|T\|$. Also \eqref{EQ:norm compute limit op} comes from \eqref{EQ:norm compute ultra}.

For the ``Moreover'' part, let $A,B\subseteq X(\omega)$ be finite subsets with $d_\omega(A,B)>R$. For $F:=A \cup B$, choose $Y$ as above. Then for each $x\in Y$, the subsets $A_x:=\{t_\alpha(x): \alpha \in A\}$ and $B_x:=\{t_\alpha(x): \alpha \in B\}$ satisfy $d(A_x, B_x) > R$, which implies $\|\chi_{A_x} T \chi_{B_x}\| < \varepsilon$. Using the complete contraction of $q_\omega$ as above, we obtain $\|\chi_A\Phi_\omega(T)\chi_B\|\leq\varepsilon<2\varepsilon$. For general $A,B \subseteq X(\omega)$ with $d_\omega(A,B)>R$, note that
\[
\|\chi_A\Phi_\omega(T)\chi_B\| = \sup\{\|\chi_{A'}\Phi_\omega(T)\chi_{B'}\|: \text{finite } A' \subseteq A \text{ and finite } B' \subseteq B\}.
\]
So we finish the proof.
\end{proof}

Focusing on quasi-local operators, we have the following:

\begin{lem}\label{lem:hom of limit op}
For every $\omega\in\partial\beta X$, the map
\[
 \Phi_\omega: C_q^*(X;\H_0)\longrightarrow
 \B\big(\ell^2(X(\omega);\K_\omega)\big)
\]
is a $C^*$-homomorphism.
\end{lem}

\begin{proof}
The proof is the Hilbert-coefficient version of \cite[Section 4]{GQW24}. For the readers' convenience, here we provide the proof of multiplicativity.

Fix $T,S\in C_q^*(X;H_0)$ and $\alpha,\beta\in X(\omega)$. Given $\varepsilon>0$, choose $R>0$ such that
\[
 \|\chi_{\{z\}}S\chi_{X\setminus B(z,R)}\|<\varepsilon \quad \text{for any} \quad z \in X.
\]
Put $F:=B_{X(\omega)}(\alpha,R)$. By \cite[Proposition 3.9 and Remark 3.12]{SW17}, there exists $Y \in \omega$ such that for any $x\in Y$, we have
\[
\left\{t_\gamma(x):\gamma\in F\right\} = B_X(t_\alpha(x),R).
\]
Hence we have
\begin{align*}
(ST)_{t_\alpha(x),t_\beta(x)} &= \chi_{\{t_\alpha(x)\}} S \chi_{B_X(t_\alpha(x),R)} T \chi_{\{t_\beta(x)\}} + \chi_{\{t_\alpha(x)\}} S \chi_{X \setminus B_X(t_\alpha(x),R)} T \chi_{\{t_\beta(x)\}} \\
& =  \sum_{\gamma \in F} S_{t_\alpha(x), t_\gamma(x)} T_{t_\gamma(x), t_{\beta}(x)} + \chi_{\{t_\alpha(x)\}} S \chi_{X \setminus B_X(t_\alpha(x),R)} T \chi_{\{t_\beta(x)\}}.
\end{align*}
Since the sum is finite, passing to ultrapower gives
\[
 \left\|
 \big(\Phi_\omega(ST)\big)_{\alpha,\beta}
 -\big(\Phi_\omega(S)\chi_F\Phi_\omega(T)\big)_{\alpha,\beta}
 \right\|
 \leq\varepsilon\|T\|.
\]
On the other hand, Lemma~\ref{lem:matrix boundedness} shows $\|\chi_{\{\alpha\}}\Phi_\omega(S)\chi_{X(\omega)\setminus F}\|<2\varepsilon$.
Combining them together, we obtain
\[
\left\|
 \big(\Phi_\omega(ST)\big)_{\alpha,\beta}
 -\big(\Phi_\omega(S)\Phi_\omega(T)\big)_{\alpha,\beta}
 \right\|
 \leq3\varepsilon\|T\|.
\]
Letting $\varepsilon \to 0+$, we conclude the proof.
\end{proof}

Finally, we record the following, which generalises both \cite[Proposition 8.2]{SW17} and a similar result in \cite{GQW24}.

\begin{prop}\label{prop:ghost intersection of limit op}
We have $G_q(X;\H_0)=\bigcap_{\omega\in\partial\beta X} \ker(\Phi_\omega)$.
\end{prop}

\begin{proof}
If $T$ is a ghost, the points $t_\alpha(x)$ and $t_\beta(x)$ eventually leave every finite subset of $X$ for any $\omega \in \partial \beta X$ and $\alpha,\beta\in X(\omega)$.  Hence \eqref{EQ:norm compute limit op} shows that $(\Phi_\omega(T))_{\alpha,\beta}=0$, which implies $\Phi_\omega(T)=0$.

Conversely, suppose that $T$ is not a ghost. Then there exist $\varepsilon>0$ and a sequence $\{(x_n,y_n)\}_n \subseteq X \times X$ tending to infinity such that $\|T_{x_n,y_n}\| \geq \varepsilon$. Since $T$ is quasi-local, there exists $R>0$ such that $\|T_{x,y}\|<\varepsilon$ for any $x,y\in X$ with $d(x,y) >R$. Hence $(x_n,y_n) \in E_R:=\{(x,y) \in X \times X: d(x,y) \leq R\}$ for all $n\in \NN$. Decomposing $E_R$ into finitely many partial translations, we choose a partial translation $t:D \to R$ whose graph is contained in $\{(x_n,y_n)\}_n$ and hence, $\|T_{t(x),x}\| \geq \varepsilon$ for any $x\in D$. Taking $\omega \in \overline D \setminus D$, \eqref{EQ:norm compute limit op} shows that $\Phi_\omega(T)_{t(\omega),\omega} \neq 0$. Therefore, $\Phi_\omega(T) \neq 0$.
\end{proof}

As a direct corollary, we obtain:

\begin{cor}\label{cor:qnorm}
For every $T\in C_q^*(X;\H_0)$, we have $\|T+G_q(X;\H_0)\| =\sup_{\omega\in\partial\beta X}\|\Phi_\omega(T)\|$.
\end{cor}

Back to the setting of Theorem~\ref{thm:quotient isom intro}, let $\{X_n\}_n$ be a sequence of finite graphs with bounded valency and large girth, and $X$ be their coarse disjoint union. Replacing the distances between distinct components if necessary, we assume that the metric on \(X\) is integer-valued. Note that this does not change the algebras.
The limit spaces can be computed as follows as a direct consequence of \cite[Proposition 3.10]{SW17}. The proof is left to the readers.

\begin{lem}\label{lem:limit spaces of large girth}
Let $\{X_n\}_n$ be a sequence of finite graphs with large girth and valency bounded by $D$, and $X$ be their coarse disjoint union. Then the limit space $X(\omega)$ with the metric $d_\omega$ is a tree with valency bounded by $D$ for any $\omega \in \partial \beta X$.
\end{lem}

This leads us to study multipliers on trees in the next subsection.

\subsection{Uniform truncation on trees}\label{ssec:HSS}

First, let us recall the notion of Schur multiplier briefly. For a non-empty set $Y$ and a function $\psi: Y\times Y \to \CC$, if for any $A=(A_{x,y})_{x,y\in Y} \in \B(\ell^2(Y))$, the matrix $M(A) := (\psi(x,y)A_{x,y})_{x,y\in Y}$ defines an element in $\B(\ell^2(Y))$, then we call the map $M$ a \emph{Schur multiplier}. It follows from \cite[Theorem 5.1]{Pis01} (essentially due to Grothendieck) that the Schur multiplier $M$ is completely bounded on $\B(\ell^2(Y))$ with $\|M\|_{\mathrm{cb}} = \|M\|$.

Now we consider the regular tree $\T_{q}$ with degree $q$ for $3 \leq q < \infty$. The edge-path metric $d$ gives a conditionally negative definite function on $\T_{q}$. In fact, choose a root in $\T_{q}$ and let $b(x)$ be the characteristic function in $\ell^2(\T_{q})$ of the vertices on the unique path from the root to $x$. Then $\|b(x)-b(y)\|^2=d(x,y)$. Given $s\in (0,1)$, Schoenberg's theorem makes $s^d$ positive definite on $\T_{q}$. Then the formula
\[
\left(M_{s}^{\T_{q}}(A)\right)_{x,y} := s^{d(x,y)} A_{x,y} \quad \text{for} \quad A \in \B(\ell^2(\T_{q}))
\]
defines a Schur multiplier $M_{s}^{\T_{q}}$ on $\B(\ell^2(\T_{q}))$, which is unital completely positive.

In general, given a function $\psi: \NN_0 \to \CC$ where $\NN_0:=\NN \cup \{0\}$, define $\tilde{\psi}$ by:
\[
\tilde{\psi}(x,y):=\psi(d(x,y)) \quad \text{for} \quad x,y\in \T_{q},
\]
and the map $M_{\psi}^{\T_{q}}$ on $\B(\ell^2(\T_{q}))$ by
\begin{equation}\label{EQ:Schur multiplier}
\left(M_{\psi}^{\T_{q}}(A)\right)_{x,y} := \tilde{\psi}(x,y) A_{x,y} \quad \text{for} \quad A \in \B(\ell^2(\T_{q})).
\end{equation}

We record the following useful tool:

\begin{prop}[{\cite[Theorem 1.2]{HSS10}}]\label{prop:HSS}
For $\psi: \NN_0 \to \CC$, the map $M_{\psi}^{\T_{q}}$ provides a Schur multiplier on $\B(\ell^2(\T_{q}))$ \emph{if and only if} the Hankel matrix
\[
H_{\psi}:=(\psi(i+j)-\psi(i+j+2))_{i,j\in \NN_0}
\]
is in the trace class $\mathcal{S}_1(\ell^2(\NN_0))$. In this case, we have
\[
\left\|M_{\psi}^{\T_{q}}\right\| = |c_+| + |c_-| + \left( 1-\frac{1}{q-1} \right) \cdot \left\| \left(\Id - \frac{\tau}{q-1}\right)^{-1}H_\psi \right\|_1,
\]
where
\[
c_\pm := \frac{1}{2} \lim_{n\to \infty} \psi(2n) \pm \frac{1}{2} \lim_{n\to \infty} \psi(2n+1)
\]
and $\tau$ is the operator on $\mathcal{S}_1(\ell^2(\NN_0))$ given by $\tau(A):=SAS^*$ for the unilateral shift operator $S$ on $\ell^2(\NN_0)$.
\end{prop}

Back to our case, we have:

\begin{lem}\label{lem:HSS to our case}
For $N\geq 0$ and $s\in (0,1)$, denote functions $\phi_{s,N}, \psi_{s,N}:\NN_0 \to \CC$ by
\[
\phi_{s,N}(m):=
\begin{cases}
s^m, & \text{if } m\leq N;\\
0, & \text{otherwise},
\end{cases}
\]
and
\[
\psi_{s,N}(m):=s^m - \phi_{s,N}(m) \quad \text{for} \quad m \in \NN_0.
\]
Then $M_{\psi_{s,N}}^{\T_{q}}$ is a Schur multiplier on $\B(\ell^2(\T_{q}))$ and we have
\[
\left\|M_{\psi_{s,N}}^{\T_{q}}\right\| \leq 2(Ns^{N+1}+(N+1)s^{N+2}+(1-s^2)\sum_{m>N}(m+1)s^m),
\]
which converges to $0$ as $N \to +\infty$. Hence, $M_{s,N}^{\T_{q}}:=M_{\phi_{s,N}}^{\T_{q}}$ is a Schur multiplier on $\B(\ell^2(\T_{q}))$ and we have
\[
\left\|M_{s,N}^{\T_{q}} - M_s^{\T_{q}}\right\| \to 0 \quad \text{as} \quad N \to +\infty.
\]
\end{lem}

\begin{proof}
Denote the Hankel matrix $H_N:=H_{\psi_{s,N}}$. For $m \in \NN$, denote the matrix $A_m=(a_{i,j})_{i,j\in \NN_0}$ such that $a_{i,j}=1$ if and only if $i+j=m$, and $0$ otherwise. Also set $A_{-1}:=0$. It is routine to check that $\|A_m\|_1=m+1$ for $m\in \NN_0$ and
\[
H_N=-s^{N+1}A_{N-1}-s^{N+2}A_N+(1-s^2)\sum_{m=N+1}^\infty s^mA_m.
\]
Hence
\[
\|H_N\|_1\leq Ns^{N+1}+(N+1)s^{N+2}+(1-s^2)\sum_{m>N}(m+1)s^m \to 0 \quad \text{as} \quad N \to +\infty,
\]
as the series $\sum_{m\ge0}(m+1)s^m$ converges.

Since $\tau$ is contractive on $\mathcal{S}_1(\ell^2(\NN_0))$, it is easy to see that
\[
\left\|\left(\Id-\frac{\tau}{q-1}\right)^{-1} H_N\right\|_1 \leq \left(1-\frac{1}{q-1}\right)^{-1} \|H_N\|_1 \leq 2\|H_N\|_1.
\]
Also $\lim_{m\to \infty} \psi_{s,N}(2m) = \lim_{m\to \infty} \psi_{s,N}(2m+1) = 0$. Hence the result follows directly from Proposition~\ref{prop:HSS}.
\end{proof}

Now consider a simplicial tree $Y$ with the edge-path metric $d$. For $s\in (0,1)$, similarly we have a Schur multiplier $M_s^Y: \B(\ell^2(Y)) \to \B(\ell^2(Y))$ by
\[
\left(M_s^Y(A)\right)_{x,y}=s^{d(x,y)}A_{x,y} \quad \text{for any} \quad A \in \B(\ell^2(Y)),
\]
which is unital completely positive. For $N\geq 0$ and $s\in (0,1)$, define $\phi_{s,N}$ as in Lemma~\ref{lem:HSS to our case} and the map $M^Y_{s,N}:=M^Y_{\phi_{s,N}}$ as in \eqref{EQ:Schur multiplier}. We have:

\begin{cor}\label{cor:HSS to our case}
For any simplicial tree $Y$ with bounded valency, the map $M^Y_{s,N}$ is a Schur multiplier on $\B(\ell^2(Y))$. Moreover, for fixed $s\in (0,1)$ we have
\[
\sup\left\{\left\| M^Y_s - M^Y_{s,N} \right\|: Y \text{ is a simplicial tree with bounded valency}\right\} \to 0
\]
as $N \to +\infty$.
\end{cor}

\begin{proof}
For any simplicial tree $Y$ of valency bounded by $D$, there is an isometric embedding $\iota: Y \hookrightarrow \T_D$, i.e., for any $x,y\in Y$ we have
\[
d_Y(x,y) = d_{\T_D}(\iota(x),\iota(y)).
\]
Without loss of generality, we can always take $D \geq 3$.
Hence by definition, $M_{s,N}^Y$ is the restriction of $M_{s,N}^{\T_D}$ on $\B(\ell^2(Y)) \subseteq \B(\ell^2(\T_D))$ and the same for $M_s^Y$. This implies that $M_{s,N}^Y$ is a Schur multiplier and we have
\[
\left\|M_s^Y-M_{s,N}^Y\right\| \leq \left\|M_s^{\T_D}-M_{s,N}^{\T_D}\right\|.
\]
Therefore, the result holds directly from Lemma~\ref{lem:HSS to our case}, where the speed of the convergence is independent of $D$ as shown therein.
\end{proof}

Finally, let us deal with the coefficients. Let $Y$ be a simplicial tree with the edge-path metric $d$, and $\K$ be a Hilbert space.
For $s\in (0,1)$, define a map $M_s^{Y,\K}: \B(\ell^2(Y;\K)) \to \B(\ell^2(Y;\K))$ by
\[
\left(M_s^{Y,\K}(A)\right)_{x,y}=s^{d(x,y)}A_{x,y} \quad \text{for} \quad A \in \B(\ell^2(Y;\K)),
\]
which is unital completely positive as before. For $N\geq 0$ and $s\in (0,1)$, define $\phi_{s,N}$ as in Lemma~\ref{lem:HSS to our case} and the map $M^{Y,\K}_{s,N}$ by
\[
\left(M_{s,N}^{Y,\K}(A)\right)_{x,y} := \phi_{s,N}(d(x,y)) A_{x,y} \quad \text{for} \quad A \in \B(\ell^2(Y;\K)).
\]
Since Schur multipliers are completely bounded, we have $M_{s,N}^{Y,\K}(A) \in \B(\ell^2(Y;\K))$ and moreover,
\[
\left\| M^{Y,\K}_s - M^{Y,\K}_{s,N}\right\| \leq \left\| M^{Y}_s - M^{Y}_{s,N} \right\|.
\]
Hence Corollary~\ref{cor:HSS to our case} immediately implies the following:

\begin{cor}\label{cor:HSS to our case new}
For any tree $Y$ with bounded valency and Hilbert space $\K$, the map $M^{Y,\K}_{s,N}$ is a bounded operator on $\B(\ell^2(Y;\K))$. Moreover, for fixed $s\in (0,1)$ we have
\[
\sup\left\{\left\| M^{Y,\K}_s - M^{Y,\K}_{s,N} \right\|: Y \text{ is a tree with bounded valency}, \K \text{ is a Hilbert space}\right\} \to 0
\]
as $N \to +\infty$.
\end{cor}

\subsection{Proof of Theorem~\ref{thm:quotient isom intro}}\label{ssec:proof of quotient isom}

We need the following auxiliary result.

\begin{lem}\label{lem:uniform quasilocal and Roe}
Let $\{Y_j\}_j$ be a family of simplicial trees with valency at most $D$ and operators $A_j \in \B(\ell^2(Y_j;\K_j))$ with
$\sup_j \|A_j\|$ finite, where $\K_j$ are Hilbert spaces. Assume $A_j$ are uniformly quasi-local, \emph{i.e.}, for any $\varepsilon>0$, there is $S>0$ such that each $A_j$ is $(\varepsilon,S)$-quasi-local. Then we have
\[
 \sup_j \left\|M_s^{Y_j,\K_j}(A_j)-A_j\right\| \to 0 \quad \text{as} \quad s \to 1-.
\]
\end{lem}

\begin{proof}
First, it is a standard result that $\{Y_j\}_j$ has uniform Property A (\emph{e.g.}, they have asymptotic dimension at most one uniformly). Using a family version of \cite[Theorem 3.3]{SZ20} (either by chasing the parameters directly in the proof of \cite[Theorem 3.3]{SZ20} or by consulting the proof of \cite[Theorem 5.1]{JZZ23} for more details), we obtain that for any $\varepsilon>0$, there exist $R>0$ and $B_j \in \B(\ell^2(Y_j;\K_j))$ with $\ppg(B_j) \leq R$ for each $j$ such that $\sup_j\|A_j-B_j\|<\varepsilon$. Hence
\[
\left\|A_j-M_s^{Y_j,\K_j}(A_j)\right\| \leq 2\varepsilon + \left\|B_j-M_s^{Y_j,\K_j}(B_j)\right\|.
\]
Since $\left( B_j-M_s^{Y_j,\K_j}(B_j) \right)_{x,y} = (1-s^{d(x,y)})\cdot (B_j)_{x,y}$, the Schur test implies that
\begin{align*}
\left\|B_j-M_s^{Y_j,\K_j}(B_j)\right\| &\leq \max\left\{\sup_{x\in Y_j} \sum_{y\in Y_j} \left\|\left( B_j-M_s^{Y_j,\K_j}(B_j) \right)_{x,y}\right\|, \sup_{y\in Y_j} \sum_{x\in Y_j} \left\|\left( B_j-M_s^{Y_j,\K_j}(B_j) \right)_{x,y}\right\|\right\} \\
&\leq (D+1)^R\cdot (1-s^R) \cdot \|B_j\|.
\end{align*}
This concludes the proof.
\end{proof}

Finally, we manage to prove Theorem~\ref{thm:quotient isom intro}.

\begin{proof}[Proof of Theorem~\ref{thm:quotient isom intro}]
Since $G(X;\H_0)=C^*(X;\H_0)\cap G_q(X;\H_0)$, the induced map on the quotients is injective. It remains to prove surjectivity. 

Fix $T\in C_q^*(X;\H_0)$. Lemma~\ref{lem:limit spaces of large girth} shows that all limit spaces are trees of uniformly bounded valency, and Lemma~\ref{lem:matrix boundedness} shows that $\{\Phi_\omega(T)\}_\omega$ are uniformly quasi-local. 

For $0<s<1$ and $N\geq0$, define an operator $T_{s,N}$ on $\ell^2(X;\H_0)$ by
\[
 (T_{s,N})_{x,y}
 =\phi_{s,N}(d(x,y))T_{x,y} \quad \text{for any} \quad x,y\in X,
\]
where $\phi_{s,N}$ is the function defined in Lemma~\ref{lem:HSS to our case}. Then $T_{s,N}$ is a bounded and locally compact operator with propagation at most $N$. Moreover, for any $\omega \in \partial \beta X$, take a compatible family $\{t_\alpha:D_\alpha\to R_\alpha\}_{\alpha\in X(\omega)}$. For any $\alpha, \beta \in X(\omega)$, we have
\[
\left(\Phi_\omega(T_{s,N})\right)_{\alpha, \beta} = \Big[\Big(
 \phi_{s,N}(d(t_\alpha(x),t_\beta(x)))
 T_{t_\alpha(x),t_\beta(x)}
 \Big)\Big]_\omega.
\]
Since the metric takes value only in $\NN_0$, $d_\omega(\alpha, \beta) = \lim_{x\to \omega}d(t_\alpha(x),t_\beta(x))$ implies that there exists $Y \in \omega$ such that $d(t_\alpha(x),t_\beta(x)) = d_\omega(\alpha, \beta)$ for any $x\in Y$. Hence
\[
\left(\Phi_\omega(T_{s,N})\right)_{\alpha, \beta}  = \phi_{s,N}(d_\omega(\alpha,\beta)) \cdot
 \big(\Phi_\omega(T)\big)_{\alpha,\beta}.
\]
This shows that
\[
\Phi_\omega(T_{s,N}) = M_{s,N}^{X(\omega),\K_\omega}(\Phi_\omega(T)).
\]

Applying Corollary~\ref{cor:qnorm}, we have
\begin{align*}
\|T_{s,N} &- T + G_q(X;\H_0)\| = \sup_{\omega \in \partial \beta X} \left\| M_{s,N}^{X(\omega),\K_\omega}(\Phi_\omega(T)) - \Phi_\omega(T) \right\| \\
&\leq \sup_{\omega \in \partial \beta X} \left\| M_{s,N}^{X(\omega),\K_\omega}(\Phi_\omega(T)) - M_{s}^{X(\omega),\K_\omega}(\Phi_\omega(T)) \right\| + \sup_{\omega \in \partial \beta X} \left\| M_{s}^{X(\omega),\K_\omega}(\Phi_\omega(T)) - \Phi_\omega(T) \right\|.
\end{align*}
Given $\varepsilon>0$, Lemma~\ref{lem:uniform quasilocal and Roe} gives $s\in (0,1)$ such that
\[
 \sup_{\omega \in \partial \beta X} \left\|M_s^{X(\omega),\K_\omega}(\Phi_\omega(T)) - \Phi_\omega(T)\right\| \leq \frac{\varepsilon}{2}.
\]
For such $s$, Corollary~\ref{cor:HSS to our case new} gives $N \in \NN$ such that
\[
\sup_{\omega \in \partial \beta X} \left\| M_{s,N}^{X(\omega),\K_\omega}(\Phi_\omega(T)) - M_{s}^{X(\omega),\K_\omega}(\Phi_\omega(T)) \right\| \leq \frac{\varepsilon}{2}.
\]
Combining them together, $\|T_{s,N}-T + G_q(X;\H_0)\|\leq \varepsilon$ for $s,N$ above.

Consequently, the map
\[
C^*(X;\H_0)/G(X;\H_0) \to C^*_{q}(X;\H_0)/G_{q}(X;\H_0)
\]
induced by the inclusion $C^*(X;\H_0) \hookrightarrow C^*_q(X;\H_0)$ has dense image, which implies that it is surjective. This concludes the proof.
\end{proof}

\section{Rank profiles}\label{sec:rank profiles}

In this section, we introduce the notion of rank function for operators in the quasi-local algebras and aim to prove Theorem~\ref{thm:main result K-theory}.

Throughout this section, let $\{X_n\}$ be a sequence of finite graphs with bounded valency, and $X$ be their coarse disjoint union. For each $n\in \NN$, set $d_n:=|X_n|$. The block-diagonal cutting map
\[
C^*_{uq}(X) \longrightarrow \prod_n \M_{d_n}(\CC) \big/ \bigoplus_n \M_{d_n}(\CC), \quad T \mapsto [(\chi_{X_n}T\chi_{X_n})_n]
\]
is a $\ast$-homomorphism by \cite[Lemma 3.14]{BCZ23}. Hence it induces the following map
\[
\rank: K_0(C^*_{uq}(X)) \longrightarrow K_0\left( \prod_n \M_{d_n}(\CC) \big/ \bigoplus_n \M_{d_n}(\CC) \right) \hookrightarrow \prod_n \ZZ \big/ \bigoplus_n \ZZ,
\]
where the last inclusion is given by the six-term exact sequence in $K$-theory with the fact that $K_1(\bigoplus_n \M_{d_n}(\CC)) = 0$ and $K_0(\prod_n \M_{d_n}(\CC))\hookrightarrow \prod_n \ZZ$ by taking block-ranks.
In particular, we have 
\[
\rank([(P_n)_{n\in \NN}]) = [(\rank(P_n))_{n\in \NN}]
\]
for block-diagonal projection $(P_n)_n \in C^*_{uq}(X)$. Here $\rank(\cdot)$ is the rank of the matrix.

In the following, we compute the image of ghost projections in the uniform Roe and quasi-local algebras under the $\rank$ function to distinguish their $K$-theories.

\subsection{The uniform Roe algebra case}\label{ssec:rank profile Roe case}

For ghost projections in the uniform Roe algebra, we have the following power saving result on their rank profiles.

\begin{prop}\label{prop:Roe ghost}
Let $\{X_n\}$ be a sequence of finite graphs with bounded valency and $\lim_{n\to \infty} |X_n| = \infty$, and $X$ be their coarse disjoint union. Assume there exists $\gamma>0$ such that $\girth(X_n)\geq\gamma\log |X_n|$ for sufficiently large $n$. Then for any $z\in K_0(G_u(X))$, there are $C<\infty$ and $\alpha>0$ such that $\rank((\iota\circ i_u)_\ast(z))$
has a representative $(s_n)_n \in \prod_n \ZZ$ satisfying
\[
 |s_n|\leq C |X_n|^{\,1-\alpha} \quad \text{for any} \quad n\in \NN.
\]
Here $i_u: G_u(X) \to C^*_u(X)$  and $\iota: C^*_u(X) \to C^*_{uq}(X)$ are the inclusion maps.
\end{prop}

To prove Proposition~\ref{prop:Roe ghost}, we need the lifting technique from \cite{WY12}. For each $n\in \NN$, let $\pi_n:\widetilde X_n\to X_n$ be the universal cover and $\Gamma_n:=\pi_1(X_n)$. The condition of large girth implies that for any $R>0$, there exists $N \in \NN$ such that for all $n \geq N$, the map $\pi_n$ is an isometry when restricted on any ball with radius $R$. 

According to \cite[Lemma 3.8]{WY12}, there exists a $\ast$-homomorphism
\[
\phi: \CC_u[X] \longrightarrow \prod_n \CC_u[\widetilde X_n]^{\Gamma_n}\big/ \bigoplus_n \CC_u[\widetilde X_n]^{\Gamma_n},
\]
where $\CC_u[\widetilde X_n]^{\Gamma_n}$ denotes the $\ast$-algebra of all finite propagation operators which are $\Gamma_n$-equivariant. More precisely, disregarding finitely many pieces, we consider a block-diagonal operator $T=(T_n)_n \in \CC_u[X]$ with $\ppg(T)=R$. The element $\phi(T)$ can be defined as follows: Let $N \in \NN$ such that for any $n \geq N$, the map $\pi_n$ is an isometry on any ball with radius $2R$. For such $n$, set $\widetilde T_n \in \CC_u[\widetilde X_n]^{\Gamma_n}$ by 
\[
(\widetilde T_n)_{x,y}:=
\begin{cases}
(T_n)_{\pi_n(x), \pi_n(y)}, & \text{if } d(x,y) \leq R;\\
0, & \text{otherwise}.
\end{cases}
\]
Then $\phi(T)$ is defined to be $[(0,\cdots,0,\widetilde T_N, \widetilde T_{N+1}, \widetilde T_{N+2}, \cdots)]$.

Furthermore, each $\widetilde X_n$ is a tree with bounded valency independent of $n$. It follows that the family $\{\widetilde X_n\}_n$ has uniform operator norm localisation property. Similar to \cite[Lemma 3.12]{WY12}, the map $\phi$ extends to a $\ast$-homomorphism
\begin{equation}\label{EQ:lifting homo}
\phi: C^*_u(X) \longrightarrow \prod_n C^*_u(\widetilde X_n)^{\Gamma_n}\big/ \bigoplus_n C^*_u(\widetilde X_n)^{\Gamma_n},
\end{equation}
where $C^*_u(\widetilde X_n)^{\Gamma_n}$ is the norm closure of $\CC_u[\widetilde X_n]^{\Gamma_n}$. 

We also need the following auxiliary result.

\begin{lem}\label{lem:K theory reduction}
For any $z \in K_0(G_u(X))$, there exist $k \in \NN$ and block-diagonal projections $p,q$ in $\M_k(G_u(X)^+)$ with $\pi_k(p) = \pi_k(q) \in \M_k(\CC)$ such that $z=[p] - [q]$. Here $\pi: G_u(X)^+ \to \CC$ is the map to the unit part. 
\end{lem}

\begin{proof}
Take projections $p',q'\in \M_k(G_u(X)^+)$ with $\pi_k(p')=\pi_k(q')$ and $z=[p']-[q']$. Set $a:=\SOT\text{-}\sum_{n} a_n$, where $a_n:=\chi_{X_n} p' \chi_{X_n}$ for each $n$. By \cite[Lemma 3.13]{BCZ23}, $a-p'$ is compact, which implies that $a^2-a$ is also compact. So $\lim_{n\to \infty}\|a^2_n - a_n\| = 0$. Hence for sufficiently large $n$, $\frac{1}{2} \notin \sigma(a_n)$ and we set $p_n:=\chi_{(1/2, \|a\|]}(a_n)$. Taking arbitrary projections $p_n$ for the finitely many remaining blocks, we obtain a block-diagonal projection $p:=\SOT\text{-}\sum_{n} p_n \in \M_k(G_u(X)^+)$ such that $p'-p$ is compact and $\pi_k(p) = \pi_k(p')$. Similarly, we have a block-diagonal projection $q\in \M_k(G_u(X)^+)$ such that $q'-q$ is compact and $\pi_k(q) = \pi_k(q')$.

Since \(p-p'\) and \(q-q'\) are compact, there exist \(m,n\in\mathbb Z\) such that $[p']-[p]=m[e]$ and $[q']-[q]=n[e]$, where \(e\) is a rank-one projection supported in some \(X_{n_0}\). Note that \(e\in G_u(X)\) is block diagonal and $z=[p]-[q]+(m-n)[e]$. If \(m-n\ge0\), replace \(p\) by \(p\oplus e^{\oplus(m-n)}\). If \(m-n<0\), replace \(q\) by \(q\oplus e^{\oplus(n-m)}\). This fulfils the task.
\end{proof}

\begin{proof}[Proof of Proposition~\ref{prop:Roe ghost}]
For $z\in K_0(G_u(X))$, Lemma~\ref{lem:K theory reduction} provides block-diagonal projections $p,q \in \M_k(G_u(X)^+)$ with $\pi_k(p) = \pi_k(q) \in \M_k(\CC)$ and $z=[p] - [q]$. Writing $p=(p_n)_n$ and $q=(q_n)_n$, set $r_n:= \rank(p_n) - \rank(q_n)$. Then $\rank((\iota\circ i_u)_\ast(z))=[(r_n)_n]$. Let $b_n:=p_n - q_n$ for each $n\in \NN$, and $b=(b_n)_n$. Then $b \in \M_k(G_u(X))$.

If $r_n \geq 0$, linear algebra shows that $\dim(\ran(p_n) \cap \ker(q_n)) \geq r_n$. Moreover, for $\xi \in \ran(p_n) \cap \ker(q_n)$, we have $b_n \xi = \xi$. A similar result holds for $r_n < 0$ as well. Hence in either case, there exists a subspace $W_n \subseteq \ell^2(X_n) \otimes \CC^k$ with $\dim(W_n) \geq |r_n|$ and $\|b_n\xi\| = \|\xi\|$ for $\xi \in W_n$.

Choose a self-adjoint block-diagonal operator $a = (a_n)_n \in \M_k(\CC_u[X])$ such that $\|a-b\| \leq \frac{1}{8}$ and $\ppg(a)\leq R$ for some $R>0$. For each $n\in \NN$ and $\xi \in W_n$, we have
\[
\|a_n \xi\| \geq \|b_n \xi\| - \|(a_n - b_n) \xi\| = \|\xi\| - \frac{1}{8} \|\xi\| = \frac{7}{8} \|\xi\|,
\]
which implies that $\langle a_n^2\xi, \xi \rangle \geq \left( \frac{7}{8} \right)^2 \|\xi\|^2$.
Hence $a_n^2$ has at least $|r_n|$ eigenvalues at least $\left( \frac{7}{8} \right)^2$. Therefore for each $m,n \in \NN$, we have
\begin{equation}\label{EQ:estimate 1}
|r_n|\cdot \left( \frac{7}{8} \right)^{2m} \leq \Tr_{kd_n}(a_n^{2m}),
\end{equation}
where $\Tr_{kd_n}$ denotes the canonical trace on the matrix algebra $\M_k(\B(\ell^2(X_n)))$.

Now $\phi$ in \eqref{EQ:lifting homo} induces a map
\[
\phi_k: \M_k(C^*_u(X)) \longrightarrow \prod_n \M_k(C^*_u(\widetilde X_n)^{\Gamma_n})\big/ \bigoplus_n \M_k(C^*_u(\widetilde X_n)^{\Gamma_n}).
\]
Take a representative $(\widetilde{a_n})_n$ for $\phi_k(a)$ (\emph{i.e.}, $\phi_k(a) = [(\widetilde{a_n})_n]$) such that $(\widetilde{a_n})_n$ and $a$ have the same propagation. As in \cite[Lemma 5.5]{WY12}, $\phi_k$ annihilates ghosts. Hence
\[
\|\phi_k(a)\| = \|\phi_k(a-b)\| \leq \|a-b\| \leq \frac{1}{8}.
\]
Therefore, there is $N \in \NN$ such that for all $n > N$, we have $\|\widetilde{a_n}\| \leq \frac{1}{4}$. 

Increasing \(N\) if necessary, we assume $\girth(X_n) > 4R$ for all $n>N$ and set 
\[
m_n:=\lfloor \frac{\girth(X_n)}{2R} \rfloor - 1.
\]
For each $x\in X_n$, choose a lift $\tilde x \in \widetilde X_n$. Then the choice of $m_n$ ensures that for each $x\in X_n$, we have
\[
(a^{2m_n}_n)_{x,x} = (\widetilde{a_n}^{2m_n})_{\tilde x,\tilde x}
\]
This implies that for $n>N$,
\begin{equation}\label{EQ:estimate 2}
\Tr_{kd_n}(a_n^{2m_n}) = \sum_{x\in X_n} \Tr_k((\widetilde{a_n}^{2m_n})_{\tilde x,\tilde x}) \leq kd_n\cdot \|\widetilde{a_n}\|^{2m_n} \leq kd_n\cdot \left( \frac{1}{4} \right)^{2m_n}.
\end{equation}
Combining \eqref{EQ:estimate 1} and \eqref{EQ:estimate 2}, we obtain that for such $n$ we have
\[
|r_n| \leq kd_n\cdot \left(\frac{2}{7}\right)^{2m_n}.
\]

On the other hand, by the choice of $m_n$ we have
\[
2m_n \geq \frac{\girth(X_n)}{R} -4.
\]
Enlarging $N$ again if necessary, the assumption provides $\girth(X_n)\geq\gamma\log d_n$ for $n>N$.
Combining them together, we have:
\[
|r_n| \leq kd_n\cdot \left(\frac{2}{7}\right)^{\frac{\gamma\log d_n}{R} -4} = Cd_n^{1-\alpha}
\]
for $C=k\left(\frac{2}{7}\right)^{-4}$ and $\alpha = \frac{\gamma}{R} \log \frac{7}{2}$ for $n>N$. Enlarging \(C\) to absorb the finitely many indices, we conclude the proof.
\end{proof}

\subsection{The uniform quasi-local algebra case}\label{ssec:rank profile quasi-local case}

Now we would like to construct a ghost projection in the uniform quasi-local algebra whose rank profile violates the control in Proposition~\ref{prop:Roe ghost}. The idea follows from Ozawa's recent work \cite{Oza25} using random subspaces (see also \cite{BSV26}). The following is a modified version of \cite[Lemma 4]{Oza25} with more refined information on dimension.

\begin{prop}\label{prop:ql subspace}
There is a universal constant $C>0$ such that for every $d,r \in \NN$ with $1 \leq r \leq d$, there is an $r$-dimensional subspace $V\subset\mathbb C^d$ such that the associated orthogonal projection $P$ onto $V$ satisfies the following: for every non-empty $E\subseteq \{1,2,\cdots,d\}$, we have
\[
 \|P\chi_E\|\le C\sqrt{\frac{r}{d}+
       \frac{|E|}{d}\log\frac{ed}{|E|}}.
\]
\end{prop}

\begin{proof}
Fix $d,r\in \NN$ with $1 \leq r \leq d$. Set
\[
 U(d):=\{W\in \M_d(\CC): W^*W=WW^*=\I_d\}
\]
with the Haar probability measure. Choose $W\in U(d)$ with this distribution and let
$U:\mathbb C^r \to \mathbb C^d$ be its first $r$ columns. Then $U^*U=I_r$, and the range of $U$ is a random
$r$-dimensional subspace and its range projection is $P=UU^*$.

Fixing a unit vector $v\in\mathbb C^r$, the vector $Uv$ is uniform on the unit sphere of $\mathbb C^d$. Also fix $E \subseteq \{1,2,\cdots,d\}$ with $|E|=k$. By coordinate symmetry and Jensen's inequality, we obtain
\begin{equation}\label{EQ:estimate on E}
\mathbb{E} (\|\chi_E Uv\|) \leq \sqrt{\mathbb{E} (\|\chi_E Uv\|^2)} = \sqrt{\frac{k}{d}}.
\end{equation}
Consider the function $f_E: x\mapsto\|\chi_Ex\|$ on the unit sphere of $\CC^d$, identified with the real sphere in $\mathbb R^{2d}$. This function is $1$-Lipschitz and hence \cite[Proposition 14.3.3]{MM02} implies that the median $\med(f_E)$ of $f_E$ satisfies:
\[
|\med(f_E) - \mathbb{E}(f_E)| \leq \frac{12}{\sqrt{2d}}.
\]
Combining with \eqref{EQ:estimate on E}, we obtain
\begin{equation}\label{EQ:median}
\med(f_E) \leq \sqrt{\frac{k}{d}} + \frac{12}{\sqrt{2d}}.
\end{equation}
Applying the measure concentration phenomenon (Levy's lemma, see \cite[Proposition 14.3.2]{MM02}) to $f_E$, we have
\[
\mathbb{P}\left( f_E > \med(f_E) + t \right) \leq 2 e^{-dt^2}
\]
for any $t\in [0,1]$. Combining with \eqref{EQ:median}, we obtain
\begin{equation}\label{EQ:single bad event prob}
\mathbb{P}\left( \|\chi_EUv\|>\sqrt{\frac{k}{d}} + \frac{12}{\sqrt{2d}} + t \right) \leq 2e^{-dt^2}.
\end{equation}

Choose a $\frac{1}{2}$-net $\N$ of the unit sphere of $\CC^r$ with $|\N| \leq 5^{2r}$.
For $k\in \{1,2,\cdots, d\}$, set
\[
t_k:=3\sqrt{\frac rd+\frac{k}{d}\log\frac{ed}{k}}.
\]
If $t_k\leq 1$, taking $t=t_k$ in \eqref{EQ:single bad event prob} and taking the union bound over all $v\in\mathcal N$ and all subsets $E$ of $\{1,2,\cdots,d\}$ with cardinality $k$, we obtain:
\begin{align*}
\mathbb{P}&\left( ~\exists~ E \subseteq \{1,\cdots,d\} \text{ with }|E|=k, \exists~ v\in \N \text{ s.t. } \|\chi_EUv\|>\sqrt{\frac{k}{d}} + \frac{12}{\sqrt{2d}} + t_k\right)\\
& \leq 2 |\N| \cdot  \binom{d}{k} \cdot e^{-dt_k^2} \leq 2\cdot 5^{2r} \cdot \left( \frac{ed}{k} \right)^k\cdot e^{-dt_k^2} = 2\cdot 5^{2r} \cdot \left( \frac{ed}{k} \right)^k\cdot e^{-9\left(r+k\log\frac{ed}{k}\right)} \\
&= 2\cdot e^{-(9-2\log 5)r} \cdot e^{-8k\log\frac{ed}{k}}  \leq 2 \cdot e^{-5r} \cdot e^{-8k},
\end{align*}
where we use the fact that $\binom{d}{k} \leq \left( \frac{ed}{k} \right)^k$ in the second inequality and $9-2\log 5 > 5$, $\log \frac{ed}{k} \geq 1$ in the last.
If $t_k > 1$, then $\|\chi_EUv\| \leq 1 <t_k$ and hence
\[
\|\chi_EUv\|\leq \sqrt{\frac{k}{d}} + \frac{12}{\sqrt{2d}} + t_k
\]
holds automatically. Thus, the corresponding bad event can never happen. Summing over $k \in \{1,2,\cdots,d\}$, we obtain:
\begin{align*}
\mathbb{P}&\left( ~\exists~ E \subseteq \{1,\cdots,d\}, \exists~ v\in \N \text{ s.t. } \|\chi_EUv\|>\sqrt{\frac{|E|}{d}} + \frac{12}{\sqrt{2d}} + t_{|E|}\right) \\
&\leq 2e^{-5r}\sum_{k=1}^{d}e^{-8k} \leq 2e^{-5r}\sum_{k=1}^{\infty}e^{-8k} = 2e^{-5r}\frac{e^{-8}}{1-e^{-8}} < 1.
\end{align*}

Consequently, there exists an isometry $U: \CC^r \to \CC^d$ such that for any non-empty $E \subseteq \{1,2,\cdots,d\}$ and any $v\in \N$, we have
\[
\|\chi_EUv\| \leq \sqrt{\frac{|E|}{d}} + \frac{12}{\sqrt{2d}} + 3\sqrt{\frac rd+\frac{|E|}{d}\log\frac{ed}{|E|}}.
\]
Since $\N$ is a $\frac{1}{2}$-net, we have
\begin{equation}\label{EQ:control for UE}
\|\chi_EU\| \leq 2\max_{v\in\N}\|\chi_EUv\| \leq 2\left(\sqrt{\frac{|E|}{d}} + \frac{12}{\sqrt{2d}} + 3\sqrt{\frac rd+\frac{|E|}{d}\log\frac{ed}{|E|}}\right).
\end{equation}
Setting
\[
S_E:=\sqrt{\frac rd+\frac{|E|}{d}\log\frac{ed}{|E|}},
\]
we have $\sqrt{\frac{|E|}{d}} \leq S_E$ and $\frac{12}{\sqrt{2d}} \leq 6\sqrt{2}S_E$. Hence \eqref{EQ:control for UE} shows
\[
\|\chi_EU\| \leq (8+12\sqrt{2})S_E \leq 32S_E.
\]
Therefore, we obtain
\[
\|P\chi_E\| = \|UU^* \chi_E\| \leq \|\chi_EU\| \leq 32 \sqrt{\frac rd+\frac{|E|}{d}\log\frac{ed}{|E|}},
\]
which concludes the proof.
\end{proof}

\begin{cor}\label{cor:ghost in ql}
Let $\{X_n\}$ be a sequence of expander graphs with $d_n:=|X_n|$, and take $r_n \in \NN$ with $1 \leq r_n \leq d_n$ and $\lim_{n\to \infty} r_n / d_n = 0$ (\emph{e.g.}, $r_n = \left\lfloor d_n/\log d_n\right\rfloor$ for large $n$). For each $n$, take a subspace $V_{n}\subseteq\ell^2(X_n)$ satisfying the result in Proposition~\ref{prop:ql subspace} for $d_n$ and $r_n$. Denote by $P_n$ the orthogonal projection in $\B(\ell^2(X_n))$ with range $V_{n}$,  and set $P:=\mathrm{(SOT)}\text{-}\sum_{n} P_{n}$.
Then $P$ is a ghost projection in $C^*_{uq}(X)$ with $\rank(P_n) = r_n$ for each $n$.
\end{cor}

\begin{proof}
It is sufficient to show that $P$ is ghost and quasi-local. By Proposition \ref{prop:ql subspace}, the proof of quasi-locality is similar to that for \cite[Theorem B]{Oza25}, while the proof that $P$ is a ghost is similar to that for Proposition~\ref{prop:Ozawa's Theorem B}. Hence we omit the details.
\end{proof}

\subsection{Proof of Theorem~\ref{thm:main result K-theory}}\label{ssec:proof of main thm III}

\begin{proof}[Proof of Theorem~\ref{thm:main result K-theory}]
Consider the following commutative diagram:
\[
\xymatrix{
0 \ar[r] & G_u(X) \ar[r]^{i_u} \ar[d]_{\iota_G} & C^*_u(X) \ar[r]^-{\pi_u} \ar[d]_{\iota} & C^*_u(X)/G_u(X) \ar[r] \ar_{\cong}[d] & 0 \\
0 \ar[r] & G_{uq}(X) \ar[r]_{i_{uq}} & C^*_{uq}(X) \ar[r]_-{\pi_{uq}} & C^*_{uq}(X)/G_{uq}(X) \ar[r] & 0.
}
\]
Without loss of generality, assume $d_n:=|X_n| \geq 2$ for each $n$, and take $r_n := \left\lfloor\frac{d_n}{\log d_n}\right\rfloor$. Then Corollary~\ref{cor:ghost in ql} provides a ghost projection $P$ in $C^*_{uq}(X)$. If $\iota_\ast$ were surjective, there would exist $y\in K_0(C^*_u(X))$ such that $\iota_\ast(y) = [P]$. Since $P$ is ghost, consider its $K$-class $[P]_G$ in $K_0(G_{uq}(X))$. Then $(i_{uq})_\ast([P]_G) = [P]$ and hence, $(\pi_{uq})_\ast([P]) = 0$. A diagram chase gives $[x] \in K_0(G_u(X))$ such that $(i_u)_\ast([x]) = [y]$, which implies
\[
\rank((\iota \circ i_u)_\ast([x])) = \rank([P]).
\]
Proposition~\ref{prop:Roe ghost} shows that the left hand side has a representative $[(s_n)_n]$ with constant $C$ and $\alpha>0$ such that $|s_n| \leq C d_n^{1-\alpha}$. On the other hand, Corollary~\ref{cor:ghost in ql} gives $s_n = \rank(P_n) = \left\lfloor\frac{d_n}{\log d_n}\right\rfloor$ for sufficiently large $n$. This leads to a contradiction since
\[
\left\lfloor\frac{d_n}{\log d_n}\right\rfloor \big/ d_n^{1-\alpha} \sim \frac{d_n^\alpha}{\log d_n} \to +\infty
\]
as $n\to \infty$. Therefore, we conclude the proof.
\end{proof}

\section{A sufficient condition on injectivity}\label{sec:injectivity}

In this section, the Hilbert space $\H_0$ is separable. The whole section is devoted to the proof of Theorem~\ref{thm:inj intro}, which is built on the following commutative diagram:
\[
\xymatrix{
0 \ar[r] & G(X;\H_0) \ar[r]^{i} \ar[d]_{\iota_G} & C^*(X;\H_0) \ar[r]^-{\pi} \ar[d]_{\iota} & C^*(X;\H_0)/G(X;\H_0) \ar[r] \ar_{\cong}[d] & 0 \\
0 \ar[r] & G_{q}(X;\H_0) \ar[r]_{i_{q}} & C^*_{q}(X;\H_0) \ar[r]_-{\pi_{q}} & C^*_{q}(X;\H_0)/G_{q}(X;\H_0) \ar[r] & 0.
}
\]
First, we apply Theorem~\ref{thm:quotient isom intro} to reduce the problem to ghost ideals.

\begin{lem}\label{lem:reduction to ghost ideals}
Let $\{X_n\}$ be a sequence of graphs with bounded valency and large girth, and $X$ be their coarse disjoint union. Then $\ker(\iota_\ast) = i_\ast(\ker (\iota_{G,\ast}))$.
\end{lem}

\begin{proof}
First, it is clear that $i_\ast(\ker (\iota_{G,\ast})) \subseteq \ker(\iota_\ast)$. Conversely, let $y\in K_\ast(C^*(X;\H_0))$ with $\iota_\ast(y) = 0$. Since the quotients are isomorphic, we obtain $\pi_\ast(y) = 0$, which implies that $y=i_\ast(x)$ for some $x\in K_\ast(G(X;\H_0))$. Also we have $\iota_{G,\ast}(x) \in \ker(i_{q,\ast})$, which implies that $\iota_{G,\ast}(x) = \partial_q(z)$ for some $z\in K_{\ast+1}(C^*_{q}(X;\H_0)/G_{q}(X;\H_0))$. Here $\partial_q$ denotes the boundary map on the quasi-local algebra level. The isomorphism of the quotient algebras identifies $z$ with some $z' \in K_{\ast+1}(C^*(X;\H_0)/G(X;\H_0))$. Hence $\partial_q(z) = \iota_{G,\ast}(\partial(z'))$ for the boundary map $\partial$ on the Roe algebra level. Letting $x':=x-\partial(z')$, we obtain $x' \in \ker(\iota_{G,\ast})$ and $y=i_\ast(x')$. 
\end{proof}

We also need the following, which essentially comes from \cite[Lemma 33]{Fin14}. Denote by $\mathfrak{K}(\ell^2(X;\H_0))$ the algebra of compact operators.

\begin{lem}\label{lem:K0 inj on compacts}
Let $\{X_n\}$ be a sequence of graphs with bounded valency, and $X$ be their coarse disjoint union. Then the map
\[
K_0(\mathfrak{K}(\ell^2(X;\H_0))) \longrightarrow K_0(C^*_{q}(X;\H_0)),
\]
induced by the inclusion $\mathfrak{K}(\ell^2(X;\H_0)) \hookrightarrow C^*_{q}(X;\H_0)$, is injective.
\end{lem}

\begin{proof}
Denote
\[
D:=C^*_q(X;\H_0) \cap \prod_n \B(\ell^2(X_n;\H_0))
\]
and 
\[
J:=D \cap \KK(\ell^2(X;\H_0)) = \bigoplus_n \KK(\ell^2(X_n;\H_0)).
\]
Moreover, \cite[Lemma 3.14]{BCZ23} shows that $D + \KK(\ell^2(X;\H_0)) = C^*_q(X;\H_0)$ when $\H_0$ is infinite-dimensional, and the proof can also be applied when $\H_0$ is finite-dimensional. Hence, the inclusion of $D$ into $C^*_q(X;\H_0)$ induces the isomorphism
\begin{equation}\label{EQ:quotient isom diag}
D/J\cong C^*_q(X;\H_0)/\KK(\ell^2(X;\H_0)).
\end{equation}

Now we show that the map $K_0(J) \to K_0(D)$, induced by the inclusion $J \hookrightarrow D$, is injective. Note that $K_0(J) \cong \bigoplus_{n\in \NN} \ZZ$, given by the rank of projections. For each $n$, consider the coordinate evaluation $\mathrm{ev}_n: D \to \mathfrak{K}(\ell^2(X_n;\H_0))$. If $x=(x_n)_n \in K_0(J)$ maps to $0$ in $K_0(D)$, apply $(\mathrm{ev}_{n})_\ast$ for each $n$ and hence, $x_n = 0$. Therefore, $x=0$.

Finally, consider the following commutative diagram:
\[
\xymatrix{
0 \ar[r] &
J \ar[r] \ar@{^{(}->}[d] &
D \ar[r] \ar@{^{(}->}[d] &
D/J \ar[r] \ar[d]^{\cong} &
0 \\
0 \ar[r] &
\KK(\ell^2(X;\H_0)) \ar[r] &
C^*_q(X;\H_0) \ar[r] &
C^*_q(X;\H_0)/\KK(\ell^2(X;\H_0)) \ar[r] &
0 .
}
\]
The injectivity proved above implies that $\partial:K_1(D/J) \to K_0(J)$ is zero. By \eqref{EQ:quotient isom diag}, $\partial:K_1(C^*_q(X;\H_0)/\KK(\ell^2(X;\H_0))) \to K_0(\KK(\ell^2(X;\H_0)))$ is also zero, which shows the injectivity we need by exactness. 
\end{proof}

Finally, we finish the proof of Theorem~\ref{thm:inj intro}.

\begin{proof}[Proof of Theorem~\ref{thm:inj intro}]
By Lemma~\ref{lem:reduction to ghost ideals}, it suffices to show that $\ker (\iota_{G,\ast}) = 0$. Now \cite[Theorem 5.4]{WZ25} and \cite[Theorem 5.3]{WFZ25} show that $K_0(\mathfrak{K}(\ell^2(X;\H_0))) \cong K_0(G(X;\H_0))$ since $X$ can be coarsely embedded into some Hilbert space. Hence Lemma~\ref{lem:K0 inj on compacts} concludes that $\ker (\iota_{G,\ast}) = 0$.
\end{proof}

\appendix

\section{A groupoid proof for the uniform case of Theorem~\ref{thm:quotient isom intro}}\label{app:groupoid}

Here we give an alternative proof of Theorem~\ref{thm:quotient isom intro} when $\H_0 = \CC$ using groupoid quasi-local algebras introduced in \cite{JZZ23} (see \cite{JZZ23} for basic notions).

We only consider a locally compact Hausdorff \'{e}tale groupoid $\G$ with unit space $\Gz$, source map $s$ and range map $r$. Set $\G_x:=s^{-1}(\{x\})$ for $x\in \Gz$. 

The space $C_c(\G)$ of compactly supported continuous functions on $\G$ forms a $*$-algebra by convolution. 
Let $L^2(\G)$ be the Hilbert $C^*$-module over $C_0(\Gz)$ obtained by taking completion of $C_c(\G)$ with respect to the $C_0(\Gz)$-valued inner product:
\[
    \langle \eta,\xi \rangle(x)\coloneqq \sum_{\gamma\in \G_{x}}\overline{\eta(\gamma)}\xi(\gamma) \quad \text{for} \quad \eta, \xi \in C_c(\G),
\]
and the right $C_0(\Gz)$-module structure is given by $(\xi f)(\gamma)\coloneqq \xi(\gamma) f(s(\gamma))$. For each $x\in\Gz$, we have a left regular representation $\lambda_x: C_c(\G) \to \B(\ell^2(\G_x))$ defined by
\[
    \lambda_x(f)(\xi)(\gamma)=\sum_{\alpha\in \G_x}f(\gamma\alpha^{-1})\xi(\alpha), \quad \text{where } f\in C_{c}(\G) \text{ and }\xi\in\ell^{2}(\G_{x}),
\]
which can be formed into a single representation $\Lambda: C_c(\G) \to \L(L^2(\G))$.
Here $\L(L^2(\G))$ denotes all adjointable operators on $L^2(\G)$. The \emph{reduced groupoid $C^*$-algebra} $C^*_r(\G)$ is the norm closure of $C_c(\G)$ with the norm $\|\Lambda(f)\|$ for $f\in C_c(\G)$.

Recall that there is a useful slicing map introduced in \cite[Formula (3.1)]{JZZ23}:
\begin{equation}\label{EQ:slicing map}
    \Phi\colon \L(L^2(\G)) \longrightarrow \prod_{x\in\Gz}\B(\ell^2(\G_{x})) \quad \text{by} \quad \Phi(T) := (\Phi_x(T))_{x\in\Gz}.
\end{equation}
Note that for $T \in C^*_r(\G)$, we have $\Phi_x(T) = \lambda_x(T)$ for each $x\in \Gz$.

A notion of quasi-locality in the setting of groupoids is also introduced in \cite[Definition 4.2]{JZZ23}, which generalises the metric space case. More precisely, an operator $T\in \L(L^2(\G))$ is called \emph{quasi-local} if for any $\varepsilon>0$, there exists a compact subset $K\subset \G$ such that for any $f,g\in C_{b}(\G)$ with $(K\cdot \supp(f))\cap \supp(g)=\emptyset$ and $\supp(f)\cap (K\cdot \supp(g))=\emptyset$, we have $\|gTf\| < \varepsilon\|g\|_{\infty}\|f\|_{\infty}$. Denote the set of all quasi-local operators in $\L(L^2(\G))$ by $C^*_{uq}(\G)$, called the \emph{quasi-local algebra of $\G$}.

Recall that $T \in \L(L^2(\G))$ is \emph{$\G$-equivariant} if it commutes with the action of right convolution of $C_c(\G)$ (see \cite[Definition 3.21]{JZZ23}). Denote the set of all $\G$-equivariant quasi-local operators by $C^*_{uq}(\G)^{\G}$. \cite[Theorem E]{JZZ23} shows that if $\G$ is topologically amenable, then 
\begin{equation}\label{EQ:isom groupoid}
C^*_r(\G) \cong C^*_{uq}(\G)^{\G}.
\end{equation}

To apply the groupoid tool to our setting, we need the notion of coarse groupoid introduced in \cite{STY02} (see also \cite[Chapter 10]{R2003-book-a}). Given a discrete metric space $(X,d)$ of bounded geometry, as a topological space, the \emph{coarse groupoid} is 
\[
    G(X):=\bigcup_{r>0}{\overline{E_r}}^{\beta (X \times X)} \subseteq \beta (X \times X),
\]
where $E_r:=\{(x,y) \in X\times X: d(x,y) \leq r\}$. Also \cite{STY02} shows that $G(X)$ has a groupoid structure, which is locally compact Hausdorff and \'{e}tale with unit space $\beta X$.  
Denote by $\partial \beta X:=\beta X \setminus X$ the Stone-\v{C}ech boundary of $X$. 
Moreover, \cite[Proposition 10.29]{R2003-book-a} shows that there is a $\ast$-isomorphism from $C_c(G(X))$ to $\mathbb{C}_u[X]$, which extends to a $C^*$-isomorphism $C^*_r(G(X)) \cong C^*_u(X)$. On the other hand, it is shown in \cite[Example 4.9 and Section 6.3]{JZZ23} that $C^*_{uq}(G(X))^{G(X)} \cong C^*_{uq}(X)$.

The following result links the groupoid algebras to the quotient algebras. 

\begin{prop}\label{prop:quotient groupoid version}
Let $(X,d)$ be a discrete metric space of bounded geometry and $\G:=G(X)|_{\partial \beta X}=s^{-1}(\partial \beta X)$, where $s: G(X) \to \beta X$ is the source map. Then we have
\[
C^*_u(X)/G_u(X) \cong C^*_r(\G)
\]
and an embedding
\[
C^*_{uq}(X)/G_{uq}(X) \hookrightarrow  C^*_{uq}(\G)^{\G}.
\]
\end{prop}

\begin{proof}
The first isomorphism is a reformulation of \cite[Example 4.5 and Proposition 4.9]{WZ25}. Hence, we only focus on the second.

For each $\omega \in \beta X$, \cite[Lemma C.3]{SW17} shows that $s^{-1}(\omega)$ coincides with the limit space $X(\omega)$. Applying  \cite[Corollary 4.6]{JZZ23} to both $C^*_{uq}(G(X))$ and $C^*_{uq}(\G)$, we see that for any $S \in C^*_{uq}(G(X))$, the family $\{(\Phi_x(S))\}_{x\in \partial \beta X}$ from the slicing map \eqref{EQ:slicing map} determines an operator $\tilde S \in C^*_{uq}(\G)$. Also \cite[Lemma 3.23]{JZZ23} shows that if $S \in C^*_{uq}(G(X))^{G(X)}$, then $\tilde S \in C^*_{uq}(\G)^{\G}$. Hence we obtain a $C^\ast$-homomorphism $\rho$ defined to be the composition
\[
\rho: C^*_{uq}(X) \stackrel{\cong}{\longrightarrow} C^*_{uq}(G(X))^{G(X)} \longrightarrow C^*_{uq}(\G)^{\G}.
\]
For $T \in C^*_{uq}(X)$ and $\omega \in \partial \beta X$, simply write $T_{\omega}:=\Phi_\omega(\rho(T)) \in \B(\ell^2(X(\omega)))$. 
Let us compute each $T_\omega$ more precisely: For $\alpha, \beta \in X(\omega)$, choose partial translations $t_\alpha: D_\alpha \to R_\alpha$ and $t_\beta: D_\beta \to R_\beta$ with $D_\alpha, D_\beta \in \omega$, $t_\alpha(\omega) = \alpha$ and $t_\beta(\omega) = \beta$. Then
\begin{equation}\label{EQ:T omega}
(T_\omega)_{\alpha, \beta} = \lim_{x\to \omega} T_{t_\alpha(x), t_\beta(x)}.
\end{equation}
Therefore, the proof of $\ker(\rho) = G_{uq}(X)$ follows exactly as in the proof of Proposition~\ref{prop:ghost intersection of limit op}. 
\end{proof}

\begin{rem}\label{rem:limit spaces and operators}
Since $\H_0 = \CC$, we use the classical limit operators from \cite{GQW24, SW17} rather than using ultrapower as in Section~\ref{sec:quotient isom}. It is clear from \eqref{EQ:T omega} that the operator $T_\omega$ above is the limit operator of $T$ at $\omega$.
\end{rem}

Back to the setting of Theorem~\ref{thm:quotient isom intro}, let $\{X_n\}_n$ be a sequence of finite graphs with bounded valency and large girth, and $X$ be their coarse disjoint union. If the groupoid $G(X)|_{\partial \beta X}$ is amenable, then \cite[Theorem E]{JZZ23} shows that \eqref{EQ:isom groupoid} holds. Combining with Proposition~\ref{prop:quotient groupoid version}, we will conclude the proof of Theorem~\ref{thm:quotient isom intro}. 
Unfortunately, Willett showed in \cite{Wil11} that if every vertex has valency at least $3$, then $X$ does not have property A. In this case, it is well-known (see, \emph{e.g.}, \cite[Proposition 5.1]{WZ25}) that the groupoid $G(X)|_{\partial \beta X}$ cannot be amenable. Hence, \cite[Theorem E]{JZZ23} cannot be directly used to prove \eqref{EQ:isom groupoid}.

\begin{prop}\label{prop:isomorphism non-amenable}
Let $\{X_n\}_n$ be a sequence of finite graphs with large girth and valency bounded by $D$, and $X$ be their coarse disjoint union. For the groupoid $\G:=G(X)|_{\partial \beta X}$, we have
\[
C^*_r(\G) \cong C^*_{uq}(\G)^{\G}.
\]
\end{prop}

\begin{proof}
Set $K_n:=\overline{E_n}^{\beta(X \times X)} \cap \G$ for each $n\in \NN_0$, which is clopen and satisfies (M.1)-(M.3) in \cite[Section 5]{JZZ23}. Define the length function $\ell: \G \to \NN_0$ by 
\[
\ell(\gamma):=\min\{n\in \NN_0: \gamma \in K_n\} \quad \text{for} \quad \gamma \in \G,
\]
which is continuous on $\G$. Following \cite[Section 5]{JZZ23}, define a metric on $X(\omega)$ by
\[
(\alpha, \beta) \mapsto \ell(\alpha \beta^{-1}) \quad \text{for} \quad \alpha, \beta \in X(\omega).
\]
By \cite{JZZ23, SW17}, this metric coincides with the limit metric $d_\omega$ from \eqref{EQ:limit metric}. Lemma \ref{lem:limit spaces of large girth} shows that $X(\omega)$ is a tree with valency bounded by $D$.

Fix $T \in C^*_{uq}(\G)^{\G}$. By \cite[Lemma 5.7]{JZZ23}, the family $\{\Phi_\omega(T)\}_{\omega \in \partial \beta X}$ from \eqref{EQ:slicing map} is uniformly quasi-local. 
On the other hand, \cite[Proposition 3.31]{JZZ23} shows that there is a left convolver $f_T\in C_b(\G)$ such that $T = \Lambda(f_T)$. For each $N \in \NN$ and $s\in (0,1)$, define the function $f_{s,N}$ on $\G$ by 
\[
f_{s,N}(\gamma):=s^{\ell(\gamma)} \chi_{K_N}(\gamma)f_T(\gamma) \quad \text{for} \quad \gamma \in \G.
\]
Since $K_N$ is compact and clopen, $f_{s,N} \in C_c(\G)$ and hence $T_{s,N}:=\Lambda(f_{s,N}) \in C^*_r(\G)$.

Fixing $\omega \in \partial \beta X$, direct calculations show that for any $\alpha, \beta \in X(\omega)$, we have
\begin{align*}
\left(\Phi_\omega(T_{s,N})\right)_{\alpha,\beta} &= f_{s,N}(\alpha\beta^{-1}) = s^{\ell(\alpha\beta^{-1})} \cdot \chi_{K_N}(\alpha\beta^{-1})\cdot f_T(\alpha\beta^{-1})\\
&= s^{d_\omega(\alpha,\beta)} \cdot \chi_{[0,N]}(d_\omega(\alpha,\beta)) \cdot \Phi_\omega(T)_{\alpha,\beta}\\
&= \left( M_{s,N}^{X(\omega)}(\Phi_\omega(T)) \right)_{\alpha,\beta}.
\end{align*}
Hence we conclude that $\Phi_\omega(T_{s,N}) = M_{s,N}^{X(\omega)}(\Phi_\omega(T))$. 
Since the slicing map \eqref{EQ:slicing map} is an isometry, we obtain
\begin{align*}
\|T_{s,N} - T\| &= \sup_{\omega \in \partial \beta X} \left\| M_{s,N}^{X(\omega)}(\Phi_\omega(T)) - \Phi_\omega(T) \right\| \\
&\leq \sup_{\omega \in \partial \beta X} \left\| M_{s,N}^{X(\omega)}(\Phi_\omega(T)) - M_{s}^{X(\omega)}(\Phi_\omega(T)) \right\| + \sup_{\omega \in \partial \beta X} \left\| M_{s}^{X(\omega)}(\Phi_\omega(T)) - \Phi_\omega(T) \right\|.
\end{align*}
Applying Corollary~\ref{cor:HSS to our case} and Lemma~\ref{lem:uniform quasilocal and Roe}, we obtain that $\|T_{s,N} - T\|$ can be made arbitrarily small, which concludes that $T \in C^*_r(\G)$.
\end{proof}

Combining this with Proposition~\ref{prop:quotient groupoid version} will complete the proof of Theorem~\ref{thm:quotient isom  intro} for the case of $\H_0=\CC$.

\bibliographystyle{plain}
\bibliography{reference_updated1}

\begin{thebibliography}{10}

\bibitem{AGS12}
Goulnara Arzhantseva, Erik Guentner, and J\'{a}n \v{S}pakula.
\newblock Coarse non-amenability and coarse embeddings.
\newblock {\em Geom. Funct. Anal.}, 22(1):22--36, 2012.

\bibitem{BCZ23}
Hengda Bao, Xiaoman Chen, and Jiawen Zhang.
\newblock Strongly quasi-local algebras and their {$K$}-theories.
\newblock {\em J. Noncommut. Geom.}, 17(1):241--285, 2023.

\bibitem{BBFKVW22}
Florent~P. Baudier, Bruno~M. Braga, Ilijas Farah, Ana Khukhro, Alessandro
  Vignati, and Rufus Willett.
\newblock Uniform {Roe} algebras of uniformly locally finite metric spaces are
  rigid.
\newblock {\em Invent. Math.}, 230(3):1071--1100, 2022.

\bibitem{BFV24}
Bruno~M. Braga, Ilijas Farah, and Alessandro Vignati.
\newblock Operator norm localization property for equi-approximable families of
  projections.
\newblock {\em J. Noncommut. Geom.}, 18(2):657--680, 2024.

\bibitem{BSV26}
Bruno~M. Braga, J{\'a}n {\v{S}}pakula, and Alessandro Vignati.
\newblock A note on the quasi-local algebra of expander graphs.
\newblock {\em Bull. Lond. Math. Soc.}, 58(1):e70252, 2026.

\bibitem{CGZ24}
Xiaoman Chen, Kun Gao, and Jiawen Zhang.
\newblock The strongly quasi-local coarse {Novikov} conjecture and {Banach}
  spaces with {Property (H)}.
\newblock {\em Kyoto J. Math.}, 64(1):31--74, 2024.

\bibitem{Eng14}
Alexander Engel.
\newblock {\em Indices of pseudodifferential operators on open manifolds}.
\newblock PhD thesis, Universit{\"a}t Augsburg, 2014.
\newblock arXiv:1410.8030.

\bibitem{Eng19}
Alexander Engel.
\newblock Rough index theory on spaces of polynomial growth and
  contractibility.
\newblock {\em J. Noncommut. Geom.}, 13(2):617--666, 2019.

\bibitem{Fin14}
Martin Finn-Sell.
\newblock Fibred coarse embeddings, a-{T}-menability and the coarse analogue of
  the {Novikov} conjecture.
\newblock {\em J. Funct. Anal.}, 267(10):3758--3782, 2014.

\bibitem{GQW24}
Liang Guo, Jin Qian, and Qin Wang.
\newblock A nonstandard analysis approach to limit operators and {Fredholmness}
  in {Roe}-like algebras.
\newblock arXiv:2412.08130, 2024.

\bibitem{HSS10}
Uffe Haagerup, Troels Steenstrup, and Ryszard Szwarc.
\newblock {Schur} multipliers and spherical functions on homogeneous trees.
\newblock {\em Int. J. Math.}, 21(10):1337--1382, 2010.

\bibitem{HRY93}
Nigel Higson, John Roe, and Guoliang Yu.
\newblock A coarse {Mayer--Vietoris} principle.
\newblock {\em Math. Proc. Cambridge Philos. Soc.}, 114(1):85--97, 1993.

\bibitem{JZZ23}
Baojie Jiang, Jiawen Zhang, and Jianguo Zhang.
\newblock Quasi-locality for {\'e}tale groupoids.
\newblock {\em Commun. Math. Phys.}, 403(1):329--379, 2023.

\bibitem{KLVZ21}
Ana Khukhro, Kang Li, Federico Vigolo, and Jiawen Zhang.
\newblock On the structure of asymptotic expanders.
\newblock {\em Adv. Math.}, 393:108073, 2021.

\bibitem{LNSZ21}
Kang Li, Piotr Nowak, J\'{a}n \v{S}pakula, and Jiawen Zhang.
\newblock Quasi-local algebras and asymptotic expanders.
\newblock {\em Groups Geom. Dyn.}, 15(2):655--682, 2021.

\bibitem{LSZ24}
Kang Li, J{\'a}n {\v{S}}pakula, and Jiawen Zhang.
\newblock Measured expanders.
\newblock {\em J. Topol. Anal.}, 16(6):917--944, 2024.

\bibitem{LVZ23}
Kang Li, Federico Vigolo, and Jiawen Zhang.
\newblock Asymptotic expansion in measure and strong ergodicity.
\newblock {\em J. Topol. Anal.}, 15(2):361--399, 2023.

\bibitem{LVZ23b}
Kang Li, Federico Vigolo, and Jiawen Zhang.
\newblock A {Markovian} and {Roe}-algebraic approach to asymptotic expansion in
  measure.
\newblock {\em Banach J. Math. Anal.}, 17(4):74, 2023.

\bibitem{LSZ23}
Kang Li, J\'{a}n \v{S}pakula, and Jiawen Zhang.
\newblock Measured asymptotic expanders and rigidity for {Roe} algebras.
\newblock {\em Int. Math. Res. Not. IMRN}, 2023(17):15102--15154, 2023.

\bibitem{LWZ25}
Xulong Lu, Qin Wang, and Jiawen Zhang.
\newblock Asymptotic expansion for groupoids and {Roe}-type algebras.
\newblock {\em J. Noncommut. Geom.}, 2026.
\newblock Published online first.

\bibitem{LPS88}
Alexander Lubotzky, Ralph Phillips, and Peter Sarnak.
\newblock {Ramanujan} graphs.
\newblock {\em Combinatorica}, 8(3):261--277, 1988.

\bibitem{MM02}
Ji{\v{r}}{\'\i} Matou{\v{s}}ek.
\newblock {\em Lectures on discrete geometry}, volume 212 of {\em Graduate
  Texts in Mathematics}.
\newblock Springer-Verlag, New York, 2002.

\bibitem{NWZ25}
Graham~A. Niblo, Nick Wright, and Jiawen Zhang.
\newblock Building weight-free {F{\o}lner} sets for {Yu}'s {Property A} in
  coarse geometry.
\newblock {\em Algebr. Geom. Topol.}, 25(5):3133--3144, 2025.

\bibitem{Oza25}
Narutaka Ozawa.
\newblock Embeddings of matrix algebras into uniform {Roe} algebras and
  quasi-local algebras.
\newblock {\em J. Eur. Math. Soc.}, 2025.
\newblock Published online first.

\bibitem{Pis01}
Gilles Pisier.
\newblock {\em Similarity problems and completely bounded maps}, volume 1618 of
  {\em Lecture Notes in Mathematics}.
\newblock Springer-Verlag, Berlin, second, expanded edition, 2001.

\bibitem{Roe88}
John Roe.
\newblock An index theorem on open manifolds. {I}.
\newblock {\em J. Differential Geom.}, 27(1):87--113, 1988.

\bibitem{Roe93}
John Roe.
\newblock Coarse cohomology and index theory on complete {Riemannian}
  manifolds.
\newblock {\em Mem. Amer. Math. Soc.}, 104(497), 1993.

\bibitem{Roe96}
John Roe.
\newblock {\em Index theory, coarse geometry, and topology of manifolds},
  volume~90 of {\em CBMS Regional Conference Series in Mathematics}.
\newblock American Mathematical Society, Providence, RI, 1996.

\bibitem{R2003-book-a}
John Roe.
\newblock {\em Lectures on coarse geometry}, volume~31 of {\em University
  Lecture Series}.
\newblock American Mathematical Society, Providence, RI, 2003.

\bibitem{STY02}
Georges Skandalis, Jean-Louis Tu, and Guoliang Yu.
\newblock The coarse {Baum--Connes} conjecture and groupoids.
\newblock {\em Topology}, 41(4):807--834, 2002.

\bibitem{SZ20}
J{\'a}n {\v S}pakula and Jiawen Zhang.
\newblock Quasi-locality and {Property A}.
\newblock {\em J. Funct. Anal.}, 278(1):108299, 2020.

\bibitem{Vig26}
Alessandro Vignati.
\newblock {\em Rigidity of {Roe}-like algebras}.
\newblock Habilitation thesis, Université Paris Cité, 2026.

\bibitem{Spa09}
J\'{a}n \v{S}pakula.
\newblock Uniform {$K$}-homology theory.
\newblock {\em J. Funct. Anal.}, 257(1):88--121, 2009.

\bibitem{SW17}
J\'{a}n \v{S}pakula and Rufus Willett.
\newblock A metric approach to limit operators.
\newblock {\em Trans. Amer. Math. Soc.}, 369(1):263--308, 2017.

\bibitem{WZ25}
Qin Wang and Jiawen Zhang.
\newblock Ideal structure of uniform {Roe} algebras: beyond {Property A}.
\newblock {\em Proc. Roy. Soc. Edinburgh Sect. A}, pages 1--39, 2025.
\newblock Published online first.

\bibitem{WFZ25}
Zhijie Wang, Benyin Fu, and Jiawen Zhang.
\newblock Fibring structures of ideals in {Roe} algebras and their
  {$K$}-theories.
\newblock {\em Sci. China Math.}, 2025.
\newblock Published online first.

\bibitem{WW20}
Stuart White and Rufus Willett.
\newblock {Cartan} subalgebras in uniform {Roe} algebras.
\newblock {\em Groups Geom. Dyn.}, 14(3):949--989, 2020.

\bibitem{Wil11}
Rufus Willett.
\newblock {Property A} and graphs with large girth.
\newblock {\em J. Topol. Anal.}, 3(3):377--384, 2011.

\bibitem{WY12}
Rufus Willett and Guoliang Yu.
\newblock Higher index theory for certain expanders and {Gromov} monster
  groups, {I}.
\newblock {\em Adv. Math.}, 229(3):1380--1416, 2012.

\bibitem{Yu00}
Guoliang Yu.
\newblock The coarse {Baum--Connes} conjecture for spaces which admit a uniform
  embedding into {Hilbert} space.
\newblock {\em Invent. Math.}, 139(1):201--240, 2000.

\bibitem{ZZ25}
Jingming Zhu and Jiawen Zhang.
\newblock A weight-free characterisation of {Yu}'s {Property A}.
\newblock {\em Bull. Lond. Math. Soc.}, 57(10):3005--3012, 2025.

\end{thebibliography}

\end{document}